\documentclass[12pt, leqno]{amsart}

\usepackage{graphicx}
\usepackage{amssymb,amsmath,amsthm,amscd,accents}
\usepackage{stackengine,scalerel}
\usepackage{mathrsfs}
\usepackage{lscape}
\usepackage{enumerate}
\usepackage{upgreek}
\usepackage[usenames,dvipsnames]{color}
\usepackage[colorlinks=true, pdfstartview=FitV,
 linkcolor=blue,citecolor=blue,urlcolor=blue]{hyperref}
\usepackage{tikz}
\tikzset{dynkdot/.style={circle,draw,scale=.38}}
\usetikzlibrary{arrows.meta,arrows}
\usepackage{bm}
\usepackage{marginnote}
\usepackage{verbatim}
\usepackage[normalem]{ulem}  
\usepackage{xspace}

\usepackage[all]{xy}

\allowdisplaybreaks[3]

\newcommand{\arxiv}[1]{\href{http://arxiv.org/abs/#1}{\texttt{arXiv:#1}}}

\newcommand{\nc}{\newcommand}
\numberwithin{equation}{section}

\newenvironment{red}{\relax\color{red}}{\relax}
\newenvironment{blue}{\relax\color{blue}}{\hspace*{.5ex}\relax}
\newenvironment{purple}{\relax\color{blue}}{\hspace*{.5ex}\relax}

\newenvironment{magenta}{\relax\color{magenta}}{\hspace*{.5ex}\relax}
\newenvironment{jaune}{\relax\color{YellowOrange}}{\hspace*{.5ex}\relax}

\newcommand{\bep}{\begin{purple}}
\newcommand{\eep}{\end{purple}}

\nc{\bey}{\begin{jaune}}
  \nc{\ey}{\end{jaune}}
\newcommand{\beb}{\begin{blue}}
\newcommand{\eb}{\end{blue}}
\newcommand{\bem}{\begin{magenta}}
\newcommand{\eem}{\end{magenta}}

\newcommand{\bero}[1]{\begin{red}{}\marginnote{\fbox{\scshape\lowercase{O}}}%
#1}

\newcommand{\berMH}[1]{\begin{red}{}\marginnote{\fbox{\scshape\lowercase{MH}}}%
#1}  

\newcommand{\berE}[1]{\begin{red}{}\marginnote{\fbox{\scshape\lowercase{E}}}%
#1}

\nc{\hs}{\hspace*}
\nc{\ms}{\mspace}

\nc{\qR}[1]{\ttq_{\mspace{-2mu}\raisebox{-.8ex}{${\scriptstyle{#1}}$}}}

\theoremstyle{plain}
\newtheorem{lemma}{Lemma}[section]
\newtheorem{proposition}[lemma]{Proposition}
\newtheorem{theorem}[lemma]{Theorem}

\newtheorem{corollary}[lemma]{Corollary}
\newtheorem{conjecture}{Conjecture}
\newtheorem*{convention}{Convention}

\theoremstyle{definition}
\newtheorem{remark}[lemma]{Remark}

\newtheorem{definition}[lemma]{Definition}

\nc{\Prop}{\begin{proposition}}
  \nc{\enprop}{\end{proposition}}
\nc{\Def}{\begin{definition}}
  \nc{\edf}{\end{definition}}

\newcommand{\Zq}{{\Z[q^{\pm1}]}}

\renewcommand{\le}{\leqslant}
\renewcommand{\ge}{\geqslant}
\renewcommand{\preceq}{\preccurlyeq}

\newcommand{\seteq}{\mathbin{:=}}

\newcommand{\soplus}{\mathop{\mbox{\normalsize$\bigoplus$}}\limits}
\newcommand{\sbcup}{\mathop{\mbox{\normalsize$\bigcup$}}\limits}

\newcommand{\tens}{\mathop\otimes}

\newcommand{\g}{\mathfrak{g}}

\newcommand{\n}{\mathfrak{n}}

\newcommand{\Q}{\mathbb{Q}}
\newcommand{\Z}{\mathbb{Z}\ms{1mu}}

\newcommand{\al}{{\ms{1mu}\alpha}}

\newcommand{\la}{\lambda}
\newcommand{\be}{{\ms{1mu}\beta}}

\newcommand{\vph}{\varphi}

\newcommand{\upal}{\upalpha}

\newcommand{\upve}{\upvarepsilon}
\newcommand{\upvp}{\upvarphi}

\newcommand{\wt}{{\rm wt}}
\newcommand{\rev}{{\rm rev}}

\newcommand{\up}{{\rm up}}

\newcommand{\ii}{ \bfi }

\newcommand{\Hom}{\operatorname{Hom}}

\newcommand{\End}{\operatorname{End}}

\newcommand{\tU}{\widetilde{U}}

\newcommand{\te}{\widetilde{e}}
\newcommand{\tf}{\widetilde{f}}

\newcommand{\hB}{\widehat{B}}

\newcommand{\hg}{\widehat{\g}}

 \newcommand{\ocalD}{\overline{\calD}}

\newcommand{\uw}{{\underline{w}}}

\newcommand{\sfC}{\mathsf{C}}
\newcommand{\sfD}{\mathsf{D}}

\newcommand{\sfP}{\mathsf{P}}
\newcommand{\sfQ}{\mathsf{Q}}
\newcommand{\sfR}{\mathsf{R}}
\newcommand{\sfS}{\mathsf{S}}

\newcommand{\sfW}{\mathsf{W}}

\newcommand{\sfc}{\mathsf{c}}
\newcommand{\sfd}{\mathsf{d}}

\nc{\bbA}{{\Zq}}

\newcommand{\bsa}{\boldsymbol{a}}

\newcommand{\bsu}{\boldsymbol{u}}
\newcommand{\bsv}{\boldsymbol{v}}

\nc{\bfA}{{\Z[q^{\pm1/2}]}}

\newcommand{\bfG}{\mathbf{G}}

\newcommand{\bfL}{\mathbf{L}}

\newcommand{\bfP}{\mathbf{P}}

\newcommand{\tbfG}{\widetilde{\bfG}}

\newcommand{\bfb}{\mathbf{b}}

\newcommand{\bfi}{\mathbf{i}}

\newcommand{\bfk}{\mathbf{k}}

\newcommand{\calA}{\mathcal{A}}

\newcommand{\calD}{\mathcal{D}}
\newcommand{\calE}{\mathcal{E}}
\newcommand{\calF}{\mathcal{F}}

\newcommand{\calK}{\mathcal{K}}

\newcommand{\calP}{\mathcal{P}}

\newcommand{\hcalA}{\widehat{\calA}}

\newcommand{\scrC}{\mathscr{C}}

\newcommand{\scrS}{\mathscr{S}}

\newcommand{\tscrS}{\widetilde{\scrS}}

\newcommand{\ttB}{\mathtt{B}}

\newcommand{\ttP}{\ms{2mu}\mathtt{P}\ms{2mu}}

\newcommand{\ttb}{\mathtt{b}}

\newcommand{\ttq}{\mathtt{q}}

\newcommand{\rmE}{\mathrm{E}}

\newcommand{\rmG}{\mathrm{G}}

\newcommand{\rmL}{\mathrm{L}}

\newcommand{\rmP}{\mathrm{P}}

\newcommand{\To}[1][{\hspace{2ex}}]{\xrightarrow{\,#1\,}}

\newlength{\mylength}
\newcommand*{\para}{%
  \rlap{\rotatebox{-30}{\rule[.05ex]{.4pt}{.77em}}}%
  \kern.04em%
  \rlap{\kern.36em\raisebox{0.649519052835em}{\rule{.6em}{.4pt}}}%
  \rule{.6em}{.4pt}\kern-.04em%
  \rotatebox{-30}{\rule[.05ex]{.4pt}{.77em}}}

\newcommand{\isoto}[1][]{\mathop{\xrightarrow%
[{\raisebox{.3ex}[0ex][.3ex]{$\scriptstyle{#1}$}}]%
{{\raisebox{-.6ex}[0ex][-.6ex]{$\mspace{2mu}\sim\mspace{2mu}$}}}}}

\newcommand{\wl}{\sfP}
\newcommand{\rl}{\sfQ}
\newcommand{\nrl}{\sfQ^-}

\newcommand{\weyl}{\sfW}
\newcommand{\lan}{\langle}
\newcommand{\ran}{\rangle}

\newcommand{\Aqn}{A_q(\n)}
\newcommand{\Aan}{A_{\bbA}(\n)}

\newcommand{\qt}[1]{\quad\text{#1}}
\newcommand{\qtq}[1][{and}]{\quad\text{{#1}}\quad}

\newcommand{\ee}{\end{enumerate}}
\newcommand{\bitem}{\begin{itemize}}
\newcommand{\eitem}{\end{itemize}}
\newcommand{\ben}{\begin{enumerate}[{\rm (1)}]}
\newcommand{\bnum}{\begin{enumerate}[{\rm (i)}]}
\newcommand{\bnump}{\begin{enumerate}[{\rm (i)$'$}]}
\newcommand{\bna}{\begin{enumerate}[{\rm (a)}]}
\newcommand{\bnA}{\begin{enumerate}[{\rm (A)}]}
\newcommand{\bc}{\begin{cases}}
\newcommand{\ec}{\end{cases}}
{\relax\setlength{\extrarowheight}{10pt}\begin{array}}%
                                          {\end{array}}
\newcommand{\ba}{\begin{array}}
\newcommand{\ea}{\end{array}}

\newcommand{\snoi}{\smallskip \noindent}
\newcommand{\mnoi}{\medskip \noindent}
\nc{\bnoi}{\bigskip\noindent}

\nc{\eqs}[1]{\underset{\raisebox{.4ex}[.7ex][0ex]{$\scriptstyle{#1}$}}{=}}
\nc{\ang}[1]{\langle{#1}\rangle}

\nc{\Uqmz}{U_{\bbA}^-(\g)}
\nc{\hA}{\hcalA}
\nc{\cl}{\colon}
\nc{\Proof}{\begin{proof}}
\nc{\QED}{\end{proof}}
\nc{\one}{\mathsf{1}}
\nc{\stt}[1]{\{#1\}}
\nc{\set}[2]{\left\{#1\bigm|#2\right\}}
\nc{\hAz}{\hcalA_{\bbA}}
\nc{\hAzq}{\hcalA_{\Z[q]}}
\nc{\hAzp}[1][{\ge0}]{(\hcalA_{\bbA})_{#1}}
\nc{\Cor}{\begin{corollary}}
  \nc{\encor}{\end{corollary}}
\nc{\id}{\mathrm{id}}
\nc{\rmo}{{\rm(}}
\nc{\rmf}{{\rm)}}

\newenvironment{myequation}
{\relax\setlength{\arraycolsep}{1pt}\begin{eqnarray}}
{\end{eqnarray}}
\newenvironment{myequationn}
{\relax\setlength{\arraycolsep}{1pt}\begin{eqnarray*}}
{\end{eqnarray*}}

\newcommand{\eq}{\begin{myequation}}
\newcommand{\eneq}{\end{myequation}}
\newcommand{\eqn}{\begin{myequationn}}
\newcommand{\eneqn}{\end{myequationn}}

\newcommand{\Uqgm}{U_q^-(\g)}
\newcommand{\hBi}{\hB(\infty)}
\newcommand{\trmG}{\widetilde{\rmG}}
\newcommand{\Uqhm}{U_{\bfk}^-(\g)}
\newcommand{\Uqmh}{U_{\bfk}^-(\g)}

\newcommand{\es}{e^*}
\newcommand{\Es}{\rmE^\star}

\newcommand{\llf}{(\hspace{-0.5ex}(}
\newcommand{\LLf}{\left(\hspace{-1ex}\left(}
\newcommand{\rrf}{)\hspace{-0.5ex})}
\newcommand{\RRf}{\right)\hspace{-1ex}\right)}
\newcommand{\Ang}[1]{  \lan #1 \ran  }

\newcommand{\qcomb}[3]{\left[ \begin{matrix} #1 \\ #2 \end{matrix} \right]_{#3}}
\newcommand{\bnom}[1]{\begin{bmatrix}#1\end{bmatrix}}

\newcommand{\dVert}[1]{  \Vert #1 \Vert  }

\newcommand{\pair}[1]{  \llf #1 \rrf  }
\newcommand{\Bpair}[1]{  \LLf #1 \RRf  }

\newcommand{\Cg}{\scrC_\g}

\newcommand{\oprod}{\displaystyle\prod^{\xrightarrow{}}}
\newcommand{\rprod}{\prod^{\xleftarrow{}}}

\newcommand{\akete}[1][0ex]{\rule[{#1}]{0ex}{1ex}}

\newcommand{\TT}{\textbf{\textit{T}}}

\newcommand{\Seq}{\mathrm{Seq}}
\newcommand{\Aqhn}{A_{\bfk}(\n)}

\newcommand{\hAform}[1]{\bl #1\br_{\hcalA}\ms{1mu}}
\newcommand{\kform}[1]{  ( #1 )}
\newcommand{\aform}[1]{  ({ #1 })_\n  }

\newcommand{\Azn}{A_\bbA(\n)}
\newcommand{\Aznw}{A_\bbA(\n(w))}
\newcommand{\Qh}{{\bfk}}
\newcommand{\Qq}{\Q(q)}
\newcommand{\TTi}[1]{\TT_{#1}}
\newcommand{\TTiv}[1]{\TT^{\ms{3mu}\star}_{#1}}

\newcommand{\ttPi}[1]{\ttP^\ii_{#1}}

 \newcommand{\bl}{\bigl(}
\newcommand{\br}{\bigr)}
\newcommand{\Uzgm}{U_\bbA^-(\g)}

\newcommand{\Gup}{\rmG^{{\rm up}}}
\newcommand{\bGup}{\bfG^{{\rm up}}}
\newcommand{\Lup}{\rmL^{{\rm up}}}
\newcommand{\LupA}{\Lup\left(\Azn\right)}
\newcommand{\LuphA}{\hAzq}
\newcommand{\qmodLuphA}{\mod q\LuphA}
\newcommand{\Mn}{\mathbf{M}\ms{1mu}}
\newcommand{\longepito}[1][]{\xymatrix@C=4ex{{}\ar@{->>}[r]^{#1}&{}}}
 
\usepackage[colorinlistoftodos]{todonotes}

 \nc{\cble}{unmixed\xspace}
 \nc{\hL}{\LuphA}
 \nc{\cb}{\mathfrak{c}}
 \nc{\vphi}{\varphi}
 \nc{\Lemma}{\begin{lemma}}
   \nc{\enlemma}{\end{lemma}}
 \nc{\Bi}{B(\infty)}
 \nc{\hAl}{\hAzq}
 \nc{\hAq}{\hA_{\Q(q)}}
 \nc{\hAB}{\hA_B}
 \nc{\hABp}[1][{\ge m}]{(\hA_B)_{#1}}
\nc{\diPP}[3]{ \ttP^{#1, \{ #2\} }_{#3} }
\nc{\scbul}{{\scriptstyle\bullet}}
\nc{\op}{{\mathrm{op}}}
\nc{\Qqh}{\Q(q^{1/2})}
\nc{\db}{\ocalD}
\newcommand{\akew}[1][2ex]{\rule[-1ex]{#1}{0ex}}
\renewcommand{\mod}{\akew[1.2ex]\mathrm{mod}\akew[.9ex]}

\nc{\vs}{\vspace*}
\nc{\Rem}{\begin{remark}}
\nc{\enrem}{\end{remark}}

\title[Braid symmetries on bosonic extensions]{Braid symmetries on bosonic extensions}

\author[M. Kashiwara]{Masaki Kashiwara}
\thanks{The research of M.\ Kashiwara
	was supported by Grant-in-Aid for Scientific Research (B)  23K20206,  
	Japan Society for the Promotion of Science.}
\address[M. Kashiwara]{%
Kyoto University Institute for Advanced Study, Research Institute
for Mathematical Sciences, Kyoto University, Kyoto 606-8502, Japan
\& Korea Institute for Advanced Study, Seoul 02455, Korea }
\email[M. Kashiwara]{masaki@kurims.kyoto-u.ac.jp}

\author[M. Kim]{Myungho Kim}
\address[M. Kim]{Department of Mathematics, Kyung Hee University, Seoul 02447, Korea}
\email[M. Kim]{mkim@khu.ac.kr}
\thanks{The research of M.\ Kim was supported by the National Research Foundation of
Korea (NRF) Grant funded by the Korea government(MSIT)
(NRF-2022R1F1A1076214 and NRF-2020R1A5A1016126).}

\author[S.-j. Oh]{Se-jin Oh}
\thanks{ The research of S.-j.\ Oh was supported by the National Research Foundation of
	Korea (NRF) Grant funded by the Korea government(MSIT) (NRF-2022R1A2C1004045).}
\address[S.-j. Oh]{ Department of Mathematics, Sungkyunkwan University, Suwon, South Korea}
\email[S.-j. Oh]{sejin092@gmail.com}

\author[E. Park]{Euiyong Park}
\thanks{The research of E.\ Park was supported by the National Research Foundation of Korea (NRF) Grant funded by the Korea Government(MSIT)(RS-2023-00273425 and NRF-2020R1A5A1016126).}
\address[E. Park]{Department of Mathematics, University of Seoul, Seoul 02504, Korea}
\email[E. Park]{epark@uos.ac.kr}

\date{August 14, 2024}

\begin{document}

\begin{abstract}
In this paper, we introduce a family  $\{\TT_i\}_{i \in I}$ of automorphisms on the bosonic extension $\hcalA$ of arbitrary type $\g$ and show that they satisfy the braid relations.
We call them the braid symmetries on $\hcalA$. They preserve the global basis and the crystal basis of $\hcalA$. 
Using $\TT_i$'s repeatedly, we define subalgebra $\hcalA(\ttb)$ for each positive braid word $\ttb$,  which possesses PBW type basis.  
The subalgebra $\hcalA(\ttb)$ is a generalization of  the quantum coordinate ring $A_q(\n(w))$ associated with a Weyl group element $w$.
As applications, we show that the tensor product decomposition with respect to $\ttb$ of the non-negative part  $\hcalA_{\ge 0}$,  and establish an anti-isomorphism, called the twist isomorphism,  between $\hcalA(\ttb)$
and $\hcalA(\ttb^\rev)$ preserving their PBW-bases and global bases.  
\end{abstract}

\setcounter{tocdepth}{2}

\maketitle
\tableofcontents

\section*{Introduction}

Let $\g$ be a symmetrizable Kac-Moody algebra and let $U_q(\g)$ be the quantum group associated with $\g$ . We denote by $\set{\al_i}{i\in I}$ the set of simple roots of $\g$.
There is a family of automorphisms $\set{\sfS_i}{i \in I}$ on the quantum group $U_q(\g)$ that satisfy the braid relations (\cite{LusztigBook}, see also \cite{Saito94}).
We refer to them as \emph{Lusztig's braid symmetries} on $U_q(\g)$ or simply \emph{braid symmetries}.
 For the precise definition of $\sfS_i$, see \eqref{eq: Lusztig action}.
These automorphisms have several favorable properties, including  a certain kind of compatibility with the global basis (or canonical basis) of the negative half $\Uqgm$ (\cite{Lusztig96}, \cite{Kimura16}) and with the crystal basis $B(\infty)$ of $\Uqgm$ (\cite{Saito94}).
In particular, using the braid symmetries, one can construct the \emph{quantum unipotent coordinate ring $A_q(\n(w))$} associated with a Weyl group element $w$,
along with its dual PBW basis (for a comprehensive exposition, see, for example, \cite[Section 4]{Kimura12}).

The purpose of this paper is to introduce and study a family $\set{\TT_i}{i\in I}$  of algebra automorphisms on the \emph{bosonic extension} $\hcalA$ of \emph{arbitrary type} $\g$  that satisfies the braid relations.
We call the automorphisms $\TT_i$ the \emph{braid symmetries} on the bosonic extension $\hcalA$.
Note that when $\g$ is of finite type,  these automorphisms were introduced in \cite{KKOP21B, JLO2} and studied extensively in \cite{OP24}. 

\medskip
Let us briefly recall the definition of the bosonic extension $\hcalA$ associated with $\g$. The quantum coordinate ring $\Aqn$, which is isomorphic to the negative half $\Uqgm$ of $U_q(\g)$, admits a presentation by 
the Chevalley generators $\set{f_i}{i\in I}$ subject to the  \emph{$q$-Serre relations}. 
In \cite{KKOP24}, the \emph{bosonic extension $\hcalA$ associated with $\g$}   is
presented  by a set of  generators $\set{f_{i,m}}{i\in I, m \in \Z}$
subject to the $q$-Serre relations among  $\set{f_{i,m}}{i\in I}$ for a fixed $m$,
 and the \emph{$q$-boson relations} between $f_{i,m}$ and $f_{j,p}$ for different $m$ and $p$ (See \eqref{it: def of hA (a)} and \eqref{it: def of hA (b)}). 
It is shown that the subalgebra $\hcalA[m]$ generated by $\{f_{i,m} \mid i\in I\}$  is isomorphic to $\Aqn$  for each $m$, 
and the algebra $\hcalA$ is  isomorphic as a  vector space  to the tensor product of infinitely many  copies of $\Uqgm$. 
The subalgebra $\hcalA[m,m+1]$ generated by $\{f_{i,m}, f_{i, m+1} \mid i\in I\}$ is isomorphic to the \emph{$q$-deformed boson algebra} introduced in \cite{K91}, hence we call $\hcalA$ the bosonic extension.

The study of the bosonic extension $\hcalA$ was initiated by Hernandez and Leclerc  for simply-laced finite type $\g$, in connection with finite-dimensional representation theory over the untwisted quantum affine algebra $U_q'(\hg^{(1)})$ (\cite{HL15}). 
They showed that, when $\g$ is of simply-laced finite type, the algebra $\hcalA$ is isomorphic to \emph{the quantum Grothendieck ring} $\calK_t(\Cg^0)$  of the \emph{Hernandez-Leclerc category} $\Cg^0$, which is a certain skeleton category of finite-dimensional representations of the quantum affine algebra $U_q'(\hg^{(1)})$ (For quantum Grothendieck rings, see for example, \cite{Nak04,VV02,Her04}).
This result was extended to all finite types $\g$,  not necessarily symmetric in \cite{JLO1} where it was shown that the algebra $\hcalA$ of finite type $\g$  is isomorphic to the \emph{quantum virtual Grothendieck ring} introduced in \cite{KO23}. 

In \cite{KKOP24}, the algebra $\hcalA$ is defined for arbitrary symmetrizable Kac-Moody algebra $\g$, and it is shown that  the algebra $\hcalA$ possesses a distinguished basis, called the \emph{global basis}, which is analogous to the upper global basis (or dual canonical basis) of the quantum unipotent coordinate ring $\Aqn$.
The global basis is parameterized by the extended crystal $\hB(\infty)$, which is a product of infinitely many copies of the crystal $B(\infty)$ for $\Uqgm$. It turns out that if $\g$ is of simply-laced finite type, then the normalized global basis corresponds to the \emph{$(q,t)$-characters of simple modules}, a distinguished basis of
the quantum Grothendieck ring $\calK_t(\Cg^0)$ (\cite{KKOP24}).

\medskip

To gain a better understanding of $\TT_i$, it is helpful to first explore its relationship with the braid symmetries $\sfS_i$ of $U_q(\g)$. 
We will briefly explain how $\TT_i$ and $\sfS_i$ are related.
Observe that the formula for the braid symmetry $\TT_i$ resembles that of $\sfS_i$ (see \eqref{eq: Braid group actions} and \eqref{eq: Lusztig action}).  It can be understood as an extension of the automorphisms $\sfS_i$ on $U_q(\g)$ to the bosonic extension $\hcalA$ in the following sense.
First, note that the bosonic extension $\hcalA$ contains a natural subalgebra $\hcalA[m]$ for each $m$  that is isomorphic to $\Uqgm$, but not to the whole of $U_q(\g)$. Thus, $\sfS_i$ does not restrict to $\hcalA$.
Nevertheless there are subalgebras $\Uqgm[i] $ and $\Uqgm[i]^*$ of $U_q^-(\g)$ such that $\sfS_i$ induces an isomorphism between them. 
The automorphism $\TT_i$ coincides with $\sfS_i$ on the subalgebra of $\hcalA[m]$ corresponding to $\Uqgm[i]$ (see \eqref{eq: T_i S_i}).
Note that the Chevalley generator $f_i$ is outside of $\Uqgm[i]$.
Whereas the automorphism $\sfS_i$ maps the generator $f_i$ to $-e_it_i$, an element outside  of $\Uqgm$, 
the automorphism $\TT_i$ maps the generator $f_{i,m}$ to $f_{i,m+1}$ in the subalgebra $\hcalA[m+1]$, which is the \emph{next} copy of $\Uqgm$.
In other words, the fact that $\sfS_i$ is not an automorphism on $\Aqn$ is resolved by the braid symmetries $\TT_i$ on $\hcalA$.
Thus, $\TT_i$ can be viewed as a natural extension of $\sfS_i$ reflecting the transition from $\Aqn$ to $\hcalA$.
\medskip

The main results of this paper can be summarized as follows.
\begin{enumerate}
\item  The map $\TT_i$ given by the formula \eqref{eq: Braid group actions} is a well-defined algebra automorphism on the bosonic extension $\hcalA$ of arbitrary type $\g$,
and  the family $\set{\TT_i}{i\in  I}$ satisfies the braid relations associated with $\g$ (Theorem \ref{thm: T-braid}).
Consequently, for each positive braid word $\ttb$, there is an automorphism $\TT_\ttb$ on $\hcalA$.

\item The braid symmetries $\set{\TT_i}{i\in  I}$ preserve the global basis of $\hcalA$ and induce a braid group action on the extended crystal $\hB(\infty)$.

\item  For each positive braid word $\ttb$, the subalgebra $\hcalA(\ttb)\seteq\hcalA_{\ge 0} \cap \TT_\ttb(\hcalA_{<0})$ is generated by the \emph{cuspidal elements}, and the ordered products of these cuspidal elements form a basis (the PBW basis)  of $\hcalA(\ttb)$ as a vector space. 
 Here $\hcalA_{\ge 0}$  and $\hcalA_{<0}$ denote the subalgebras of $\hcalA$ generated by $\set{f_{i,m}}{i\in I, \ m\ge 0 }$ and by $\set{f_{i,m}}{i\in I, \ m< 0 }$, respectively.  
The transition matrix between the global basis and the PBW basis of $\hcalA(\ttb)$ is unitriangular.
The subalgebra $\hcalA(\ttb)$  is  analogous to the subalgebra $A_q(\n(w))$  of $\Aqn$ associated with a Weyl group element $w$.  
\item We give two applications of the braid symmetries :
\bnum
\item The multiplication gives a linear isomorphism between the subalgebra  $\hcalA_{\ge 0}$ 
  and the tensor product $\hcalA(\star, \ttb) \tens \hcalA(\ttb)$, where $\hcalA(\star, \ttb)= \TT_\ttb(\hcalA_{\ge 0})$. 
This result is analogous to the one established for $\Uqgm$  in \cite[Theorem 1.1]{Kimura16} and \cite[Proposition 2.10]{Tani17}.

\item There is an algebra anti-isomorphism $\Theta_\ttb$  between $\hcalA(\ttb)$
and $\hcalA(\ttb^\rev)$ preserving their PBW bases and global bases, where $\ttb^\rev$ denotes the reversed word of $\ttb$.
It is analogous to the \emph{quantum twist map} between $A_q(\n(w))$ and $A_q(\n(w^{-1}))$, which was  introduced in \cite{LY19} and studied in \cite{KimuOya18}.
\ee
\end{enumerate}
\medskip

Let us provide  more detailed explanation of the results.
Since the bosonic extension $\hcalA$ is defined by the presentations (see \eqref{it: def of hA (a)} and \eqref{it: def of hA (b)}), it is in principle possible to verify whether the formula \eqref{eq: Braid group actions} is compatible with these relations. However, directly verifying the $q$-Serre relations \eqref{it: def of hA (a)} can be challenging in practice. Instead, we have adopted an approach that uses the fact that the braid symmetries $\sfS_i$ are algebra homomorphisms on $U_q(\g)$, ensuring their compatibility with the $q$-Serre relations, and extends this compatibility to the braid symmetries $\TT_i$ on $\hcalA$.
In a similar way one can show the braid relations of $\set{\TT_i}{i\in I}$. 
Recall that for finite type $\g$, the well-definedness and the braid relations were verified using the isomorphism between $\hcalA$ and the quantum (virtual) Grothendieck ring, along with certain automorphisms on this ring 
(see the second-to-last paragraph of \cite[Introduction]{JLO2}). However, this approach is not applicable for general  $\g$, as such an isomorphism and associated notions are not available in the general case.

Recall that the braid symmetry $\sfS_i$ is an isomorphism between the subalgebras $\Uqgm[i]$ and $\Uqgm[i]^*$. 
 These subalgebras of $\Uqgm$ are known to be compatible with the upper global basis, and $\sfS_i$ induces a bijection between the upper global basis of  these subalgebras (\cite{Lusztig96}, \cite{Kimura16}) and a bijection between the corresponding crystal bases (\cite{Saito94}).
In Theorem \ref{thm: braid crystal}, we show that the braid symmetry $\TT_i$  induces an automorphism of the global basis of $\hcalA$ and an automorphism of the extended crystal $\hB(\infty)$.  
Hence one may regard $\TT_i$ on $\hcalA$ as an enhancement of  $\sfS_i$ on $\Uqgm$, as $\TT_i$ provides an automorphism (and thus a braid group action) on both the global basis and crystal basis, whereas $\sfS_i$ is only a bijection between subsets of global basis and crystal basis.
Note that the operators $\sfR_i$ on $\hB(\infty)$, which corresponds to $\TT_i$ on the global basis, can be described as a component-wise application of  the operator $\scrS_i$ on $B(\infty)$ corresponding to $\sfS_i$ (known as Saito reflection) composed of some powers of usual (star-) Kashiwara operators $\te_i$ and  $\tf_i^*$.   This description appeared first in \cite{Park23}, where the braid relations for $\set{\sfR_i}{i\in I}$ were established for  finite type $\g$.

Let $\ttB =\ang{r_{i}^{\pm} \mid i\in I}$ be the braid group  associated with $\g$. 
For each positive braid word $\ttb=r_{i_1}r_{i_2} \cdots r_{i_r}$ in $\ttB_\g$, we can consider the elements $\TT_{i_1} \cdots \TT_{i_{k-1}} (q_i^{1/2} f_{i_k,0})$ for each $1\le k\le r$, which are called the \emph{cuspidal elements}.
By taking the  products of these elements in decreasing order and normalizing them to be invariant under the twisted anti-involution $c$ on $\hcalA$ (see  \eqref{eq: c and sigma} for the definition of $c$), we obtain a set  $\bfP_\ii \seteq \{ \ttP^\ii(\bsu) \ |  \ \bsu \in \Z_{\ge0}^r \} $ which depends on the choice of $\ii=(i_1,\ldots,i_r)$ for $\ttb$. We then show that $\bfP_\ii$ forms a basis (called the PBW basis) of the subalgebra  $\hcalA(\ttb)\seteq\hcalA_{\ge 0} \cap \TT_\ttb(\hcalA_{<0})$ of $\hcalA$ (Corollary \ref{cor: P_i k-basis}). 
The transition matrix between the global basis and the PBW basis of $\hcalA(\ttb)$ is unitriangular with respect to a bi-lexicographic order  on $ \Z_{\ge0}^r $,  and the non-trivial entries of this matrix belong to $q\Z[q]$.
We also show that when $\g$ is of simply-laced finite type, the matrix entries belong to $q\Z_{\ge0}[q]$ 
using the result of \cite{VV03} on the positivity of the structure coefficients of $(q,t)$-characters of simple modules. 
For a Weyl group element $w=s_{i_1} \cdots s_{i_r}$ let $\ttb = r_{i_1} \cdots r_{i_r}$ be the lift of $w$ in the braid monoid. In this case, the corresponding algebra $\hcalA(\ttb)$ is isomorphic to the quantum unipotent coordinate ring $A_q(\n(w))$ associated with $w$ which is a subalgebra of $\Aqn$. Thus the subalgebra $\hcalA(\ttb)$ and its PBW basis are a natural generalization of $A_q(\n(w))$ and its dual PBW basis, reflecting the transition from $\Aqn$ to $\hcalA$.
Note that the subalgebra $\hcalA(\ttb)$ was introduced and studied in \cite{OP24} for finite type $\g$.

Recall that the multiplication in $U_q(\g)$ provides an isomorphism of vector spaces 
$$\left( \Uqgm \cap \sfS_w (U_q^{\ge 0} ) \right) \tens \left( \Uqgm \cap \sfS_w (\Uqgm) \right) \ \isoto   \Uqgm,$$
as established in \cite{Kimura16, Tani17}.
A bosonic analogue of this result is as follows (Proposition \ref{prop: iso via multiplication}): the multiplication gives the following isomorphism of vector spaces
$$\hcalA(\ttb) \tens \hcalA(\star,\ttb) =\bl\hcalA_{\ge0} \cap \TT_\ttb(\hcalA_{<0})\br\tens \TT_\ttb(\hcalA_{\ge 0}) \isoto   \hcalA_{\ge0}.$$
Note that this result also holds for the $\Z[q^{\pm1}]$-lattices  of the subalgebras.
The second application is the bosonic analogue of the \emph{quantum twist map}, which is an algebra anti-isomorphism  $\Theta_w\cl A_q(\n(w^{-1})) \isoto A_q(\n(w))$ introduced  in \cite{LY19}.  It is shown in \cite{KimuOya18} that the quantum twist map $\Theta_w$ induces a bijection between the upper global bases. 
As suggested by the definition of  $\Theta_w$, which is given as  a composition of $\sfS_w$ with an anti-automorphism and an involution on $U_q(\g)$, we  define the quantum twist map $\Theta_\ttb$ for a positive braid $\ttb$ by the composition $\Theta_\ttb \seteq \TT_\ttb \circ \star \circ \ocalD$. 
Here  $\star$ denotes the anti-involution that sends $f_{i,m}$ to $f_{i,-m}$ and $\ocalD$ is the algebra automorphism that sends $f_{i,m}$ to $f_{i,m+1}$.
Then the quantum twist map induces an algebra anti-isomorphism from $\hcalA(\ttb^\rev)$ to $\hcalA(\ttb)$ which preserves the PBW basis and global basis (Corollary \ref{cor:Theta}, Proposition \ref{prop: twist map}).

\medskip
Finally, we remark that the braid symmetries provide a method for constructing various subalgebras of $\hcalA$.
When $\g$ is of simply-laced finite type,  it is known that the bosonic extension $\hcalA$ and several of its subalgebras possess  cluster algebra structures (\cite{HL10, HL15, Qin17, KKOP24A}).  
Motivated by this observation, we expect that many subalgebras of $\hcalA$ also have cluster algebra structures. More precisely, we propose the following : 
\begin{conjecture} 
Let $m_j\in \Z$, $\ttb_j\in\ttB$ {\rm(} $1\le j\le r$ {\rm )}
and $m'_k\in\Z$, $\ttb'_k\in\ttB$  {\rm(}  $1\le k\le r'$ {\rm )}.
Then 
$$\bigcap_{1\le j\le r}\TT_{\ttb_j}\bl\hA_{\ge m_j}\br \cap\bigcap_{1\le k\le r'}\TT_{\ttb_k}\bl\hA_{\,\le m'_k}\br$$
 has a quantum
 cluster algebra structure.
\end{conjecture}
Note that this algebra has a subset of the global basis as a basis.

\medskip
This paper is organized as follows.
Section 1 is devoted to providing the necessary backgrounds on quantum groups, quantum unipotent coordinate rings, braid symmetries on $\Uqgm$, and the crystal and upper global basis.
In Section 2, we review the bosonic extensions $\hcalA$ for arbitrary type $\g$ as well as its global basis and extended crystal basis.
In Section 3, we introduce the braid symmetries $\TT_i$ on $\hcalA$ and prove their properties, including the braid relations, the preservation of bilinear forms and lattices and
the preservation of  the global basis. 
In Section 4, we focus on the subalgebras $\hcalA(\ttb)$ and their PBW basis.
In Section 5, we present two applications:  the tensor product decomposition of $\hcalA_{\ge 0}$ with respect to $\ttb$ and the quantum twist map $\Theta_\ttb$ on $\hcalA$.

\medskip
\begin{convention}  Throughout this paper, we use the following convention.
\ben
\item For a statement $\ttP$, we set $\delta(\ttP)$ to be $1$ or $0$ depending on whether $\ttP$ is true or not. In particular, we set $\delta_{i,j}=\delta(i=j)$. 
\item For a totally ordered set $J = \{ \cdots < j_{-1} < j_0 < j_1 < j_2 < \cdots \}$, write
$$
\oprod_{j \in J} A_j \seteq \cdots A_{j_2}A_{j_1}A_{j_0}A_{j_{-1}}A_{j_{-2}} \cdots, \quad
\rprod_{j \in J} A_j \seteq \cdots A_{j_{-2}}A_{j_{-1}}A_{j_0}A_{j_{1}}A_{j_{2}} \cdots. 
$$
\item For $a\in \Z \cup \{ -\infty \} $ and $b\in \Z \cup \{ \infty \} $ with $a\le b$, we set 
\begin{align*}
& [a,b] =\{  k \in \Z \ | \ a \le k \le b\}, &&  [a,b) =\{  k \in \Z \ | \ a \le k < b\}, \allowdisplaybreaks\\
& (a,b] =\{  k \in \Z \ | \ a < k \le b\}, &&  (a,b) =\{  k \in \Z \ | \ a < k < b\},
\end{align*}
and call them \emph{intervals}. 
When $a> b$, we understand them as empty sets. For simplicity, when $a=b$, we write $[a]$ for $[a,b]$. For an interval $[a,b]$, we set
$A^{[a,b]}$ to be the product of copies of a set $A$ indexed by $[a,b]$, and 
$$
\Z_{\ge0}^{\oplus [a,b]} \seteq \{ (c_a,\ldots,c_b) \ | \ c_k \in \Z_{\ge 0} \text{ and  $c_k=0$ except for finitely many $k$'s} \}. 
$$
We define $A^{[a,b)}$,  $\Z_{\ge 0}^{\oplus [a,b)}$, ..., etc.\ in a similar way. 
\item   For a preorder $\preceq$ on a set $A$,  we write $x\prec y$ when $x\preceq y$ holds but $y\preceq x$ does not hold.
  \label{conv:preorder}
\ee
\end{convention}

\section{Preliminaries}

In this section, we briefly review the basic notions about the quantum groups and quantum unipotent coordinate rings.

\subsection{Quantum groups and Quantum unipotent coordinate rings}
\label{sec: Background} 
Let $I$ be an index set and let $q$ be an indeterminate with a formal square root $q^{1/2}$. A \emph{Cartan datum} $(\sfC,\wl,\Pi,\wl^\vee,\Pi^\vee)$ consists of
\bna
\item a symmetrizable Cartan matrix $\sfC=(c_{i,j})_{i,j \in I}$, i.e., $\sfD \sfC$ is symmetric for a diagonal matrix $\sfD = {\rm diag}(\sfd_i \in \Z_{>0} \mid i \in I)$,
\item a free abelian group $\wl$, called the \emph{weight lattice},
\item $\Pi = \{\al_i \mid i \in I \} \subset \wl$ , called the set of \emph{simple roots},
 \item $\wl^\vee \seteq \Hom(\wl,\Z)$, called the co-weight lattice, 
\item $\Pi^\vee = \{ h_i \mid i \in I \} \subset \wl^\vee$, called the set of \emph{simple coroots},
\item a $\Q$-valued symmetric bilinear form $( \cdot,\cdot)$  on $\wl$,
\ee
satisfying the standard properties  (e.g.\ see \cite{KKOP24}). 
In this paper, we take 
$\sfd_i \seteq(\al_i,\al_i)/2\in\Z_{>0}$  $(i\in I)$.
The free abelian group $\rl \seteq \bigoplus_{i \in I} \Z \al_i$ is called the \emph{root lattice} and we set $\rl^+ = \sum_{i \in I}\Z_{\ge 0} \al_i \subset \rl$ and $\rl^- = \sum_{i \in I}\Z_{\le 0} \al_i \subset \rl$.
For any $\be = \sum_{i\in I} a_i\al_i \in \rl$,
we set 
\begin{align} \label{eq: het dVert}
\Vert  \beta  \Vert =  \sum_{i\in I} |a_i|\al_i \in \rl^+.    
\end{align}  
Let $\g$ be the Kac-Moody algebra associated with a Cartan datum $(\sfC,\wl,\Pi,\wl^\vee,\Pi^\vee)$, and $\weyl$
the \emph{Weyl group} of $\g$. The Weyl group $\weyl$ is generated by the simple reflections $s_i\in {\rm Aut}(\wl)$ $(i \in I)$ defined by $s_i(\la) = \la - \lan h_i,\la \ran \al_i$ for
$\la \in \wl$. For a sequence $\ii=(i_1,\ldots,i_r) \in I^r$, we call it  a \emph{reduced sequence} of $w \in \weyl$ if $s_{i_1} \ldots s_{i_r}$ is a reduced expression of $w$.

\smallskip

We denote by $U_q(\g)$ the \emph{quantum group} over $\Q(q)$
associated with the Cartan datum $(\sfC,\wl,\Pi,\wl^\vee,\Pi^\vee)$,  which is 
generated by the Chevalley generators $e_i,f_i$ $(i\in I)$ and $q^{h}$ $(h \in \wl^\vee)$ subject to the following defining relations:
\begin{subequations}  \label{eq: quantum group rel}
\begin{gather}
q^0=1, \quad q^{h}q^{h'}=q^{h+h'}, \quad q^he_i q^{-h} = q^{\Ang{h,\upal_i}}e_i, \quad
q^hf_i q^{-h} = q^{-\Ang{h,\upal_i}}f_i,  \label{eq: quantum group rel 1} \\
e_i f_j = f_je_i +\dfrac{t_i-t_i^{-1}}{q_i-q_i^{-1}}, \qt{where
  $t_i \seteq q^{\frac{(\al_i,\al_i)}{2}h_i}$,}    \label{eq: quantum group rel 2}
  \\
 \sum_{r=0}^{b_{i,j}} (-1)^r \qcomb{b_{i,j}}{r}{i}e_i^{b_{i,j}-r}
e_je_i^r =
\sum_{r=0}^{b_{i,j}} (-1)^r \qcomb{b_{i,j}}{r}{i}f_i^{b_{i,j}-r}
f_jf_i^r =0 \  \text{ for } i\ne j.   \label{eq: quantum group rel 3}
\end{gather}    
\end{subequations}
Here $b_{i,j} \seteq 1 - c_{i,j}$, and
$$
q_i \seteq q^{\sfd_i}, \quad [k]_i = \dfrac{q_i^k-q_i^{-k}}{q_i - q_i^{-1}}, 
\quad [k]_i! = \prod_{s=1}^k [s]_i \qtq \qcomb{m}{n}{i} = \dfrac{[m]_{i}!}{[n]_i![m-n]_!} 
$$
for $m \ge n \in \Z_{\ge0}$. For any $i \in I$  and $\al = \sum_{i\in I} n_i\al_i \in \rl$, we set
\begin{align} \label{eq: zeta}
\zeta_{i} = 1 -q_i^{2} \qtq \zeta^\al \seteq \prod_{i\in I} \zeta_{i}^{n_i} = \prod_{i\in I} (1-q_i^{2})^{n_i}.    
\end{align}

Note that there exists a $\Q(q)$-algebra anti-involution $* : U_q(\g) \to U_q(\g)$ defined by 
\begin{align} \label{eq: star involution}
e_i^* =e_i, \quad f_i^* =f_i \qtq (q^h)^* = q^{-h},    
\end{align}
where we write $x^*$ for the image of $x \in U_q(\g)$ under the involution $*$.

We denote by $U_q^-(\g)$ (resp. $U_q^+(\g)$) the subalgebra of $U_q(\g)$ generated by $f_i$'s (resp. $e_i$'s) and $U_{\bbA}^-(\g)$ the $\bbA$-subalgebra of
$U^-_q(\g)$ generated by $f_i^{(n)}\seteq f_i^n/[n]_i !$  
$(i \in I, \; n \in \Z_{>0})$. 

Note that $U^-_q(\g)$ admits the weight space decomposition 
$U^-_q(\g) = \soplus_{\be \in \nrl} U^-_q(\g)_\be$. For an element
$x \in U^-_q(\g)_\be$, we set $\wt(x) \seteq \be \in \rl^-$.  

Set
$$ A_q(\n) = \soplus_{\beta \in  \rl^-} A_q(\n)_\beta \quad \text{ where } A_q(\n)_\beta \seteq \Hom_{\Q(q)}(U^-_q(\g)_{\beta}, \Q(q)).$$
Let  
$$
\Ang{ \ , \ } : \Aqn \times \Uqgm \to \Q(q)
$$
be the pairing between $\Aqn$ and $\Uqgm$. 
Then  $A_q(\n)$ also has an algebra structure via the pairing 
$\Ang{ \ , \ }$ and the \emph{twisted coproduct} $\Delta_\n$ of $\Uqgm$ 
(see~\cite{LusztigBook}). We call $A_q(\n)$ the \emph{quantum unipotent coordinate ring}.

We denote by $\Aan$ the $\bbA$-submodule of
$A_q(\n)$ consisting of  $\uppsi \in A_q(\n)$ such that $\Ang{ \uppsi, U_{\bbA}^-(\g)} \subset \bbA$.
Note that $A_{\bbA}(\n)$ is a $\bbA$-subalgebra of $A_q(\n)$ . 

\smallskip 

For each $i \in I$, we denote by $\Ang{i} \in \Aqn_{-\al_i}$ the dual element of $f_i$ with respect to $\Ang{ \ , \ }$; i.e., 
\begin{align} \label{eq: angi}
\Ang{\Ang{i},f_j} =\delta_{i,j} \quad \text{ for any $i,j\in I$.}    
\end{align}

We set $\bfk \seteq \Q(q^{1/2})$ and 
$$
\Aqhn \seteq \bfk \tens_{\Q(q)} \Aqn. 
$$

\subsection{Bilinear forms} Let us first recall the bilinear form $( \ , \  )$ on $U^-_q(\g)$ introduced in~\cite{K91}. Note that it is a unique non-degenerate symmetric  bilinear form satisfying the following properties: 
\bna
\item  $\kform{xy,z}  = \kform{x\tens y,\Delta_\n(z)}$ for any $x,y,z \in \Uqgm$,
\item $\kform{1,1}  = \kform{f_i,f_i}  = 1$ for any $i \in I$, 
\item $\kform{x,y}  = 0$ if $x \in \Uqgm_\al$ and $y \in \Uqgm_\be$ such that $\al \ne \be$.
\ee

Note that there exists another non-degenerate, symmetric bilinear form $(\ , \ )_L$ on $\Uqgm$ introduced in \cite{LusztigBook}.
The bilinear forms $\kform{ \ , \ }$ and $(\ , \ )_L$ are related by
 (cf.\ \cite[Section 2.2]{L04}) 
\begin{align} \label{eq:KL}
\zeta^{\wt(x)}\kform{x,y} =  (x,y)_L, \qt{for homogeneous elements $x,y \in \Uqgm$.}
\end{align}

Let $e_i'$ (resp.\ $\es_i$) be the adjoint of the left (resp.\ right) multiplication of $f_i$ with respect to $\kform{ \ , \ }$; i.e.,
$$
\kform{ e_i'(x) , y } = \kform{ x,f_iy } \qtq \kform{ \es_i(x) , y } = \kform{ x,yf_i } \quad \text{ for any } x,y \in \Uqgm 
$$
(see \cite{K91,LusztigBook} and also~\cite[Section 2]{Kimura12}).  Also we can characterize the operators as follows:
$$
[e_i,x] = \dfrac{\es_i(x)t_i - t_i^{-1}e'_i(x)}{q_i-q_i^{-1}} \quad\text{for $x \in \Uqgm$.}
$$

For any $i,j \in I$ and $x,y \in \Uqgm$, we have
$$
e_i'\es_j =\es_j e'_i, \quad \es_i = * \circ e_i' \circ *, \quad \kform{x^*,y^*} = \kform{x,y}
$$
and 
\begin{subequations} \label{eq: e' es}
\begin{gather}
e'_i(xy)  = e'_i(x)y + q^{(\al_i,\wt(x))} x e'_i(y), \label{eq: e' 1}\\    
\es_i(xy) = x\es_i(y) + q^{(\al_i,\wt(y))} \es_i(x)y.   \label{eq: es 1} 
\end{gather}    
\end{subequations}

\smallskip

Note that there exists a $\Q(q)$-algebra isomorphism 
\begin{align} \label{eq: iota}
\iota :  \Uqgm \isoto \Aqn     
\end{align}
defined by $\iota(f_i) \seteq \zeta_i^{-1}\Ang{i}$ for any $i\in I$. Define a bilinear form $\aform{ \ , \ }$ on $\Aqn$ by 
\begin{align} \label{eq: aform}
\aform{f,g} \seteq \Ang{ f, \iota^{-1}(g)} \quad \text{ for any } f,g \in \Aqn.     
\end{align} 
Then one can see that $\aform{ \ , \ }$ is a non-degenerate
symmetric bilinear form on $\Aqn$ satisfying the following property:
\begin{align} \label{eq: zeta power} 
\Ang{\iota(u),v}=\aform{  \iota(u), \iota(v) }=\zeta^{\wt(u)} \kform{u,v}
\end{align}
for homogeneous elements $u,v \in \Uqgm$ (\cite[Lemma 3.1]{KKOP24}).

For any $i,j \in I$, we have
\begin{align} \label{eq: aform property}
\aform{\Ang{i},\Ang{j}} =\delta_{i,j} \zeta_i \qtq \aform{\Ang{ij},\Ang{ij}} = \dfrac{\zeta_i \zeta_j}{1-q^{-2(\al_i,\al_j)}} \ \ \text{ if } (\al_i,\al_j)<0,    
\end{align}
where $\Ang{ij} = \dfrac{\Ang{i} \Ang{j} - q^{-(\al_i, \al_j)   }\Ang{j} \Ang{i}   }{ 1-q^{-2(\al_i, \al_j)} }$ when $ (\al_i, \al_j)<0$. 
Hence $$\Ang{\Ang{ij},f_if_j}=1,\quad \Ang{\Ang{ij},f_jf_i}=0
\qtq \ang{i}\ang{j}=\ang{ij}+q^{-(\al_i,\al_j)}\ang{ji}.$$  
For $i \in I$ and $n \in \Z_{>0}$, we set
\begin{align} \label{eq: Ang i n}
\Ang{i^n} \seteq q_i^{n(n-1)/2} \Ang{i}^n.     
\end{align}

\subsection{Braid symmetries and dual PBW-bases} \label{subsec: B action}
We denote by $\ttB_\g = \Ang{r_i^{\pm 1}}_{i \in I}$ the braid group associated with $\g$; i.e, it is the group generated by $\{ r_i^{\pm 1} \}_{i \in I}$
subject to the following relations:
\begin{align} \label{eq: braid relation}
{\rm (i)} \ r_{i}r_{i}^{-1} = r_{i}^{-1}r_{i}=1 \qtq 
{\rm (ii)} \ \underbrace{r_{i}r_{j}\cdots}_{m_{i,j}\text{-times}} =\underbrace{r_{j}r_{i}\cdots}_{m_{i,j}\text{-times}}
\text{ for } i \ne j \in I,    
\end{align}
where 
\eq
m_{i,j} \seteq \bc
c_{i,j}c_{j,i} +2 & \text{ if } c_{i,j}c_{j,i} \le 2, \\
6  & \text{ if } c_{i,j}c_{j,i} =3, \\
\infty & \text{ otherwise.}
\ec
\label{eq: m_ij} 
\eneq
We call the relations in~\eqref{eq: braid relation}~{\rm (ii)} the \emph{braid relations} of $\ttB_\g$. 
We usually drop $_\g$ in the notation when we have no danger of confusion. We denote by $\ttB^+$ the braid monoid generated by $\{ r_i \}_{i \in I}$  and by 
$\ell(\ttb)$  the \emph{length} of $\ttb \in \ttB^+$.

\smallskip

Now we recall the braid symmetry on $\Uqgm$ and dual PBW basis theory by mainly following \cite{LusztigBook}. For $i \in I$, we set $\sfS_i \seteq T_{i,-1}'$
and $\sfS_i^{*} \seteq T_{i,1}''$, where $T_{i,-1}'$ and $T_{i,1}''$ are Lusztig's braid symmetries defined in \cite[Chapter 37]{LusztigBook}.
They are given as follows: 
\begin{subequations} \label{eq: Lusztig action}
  \begin{gather}
   \ba{ll}
\sfS_i(t_i)=t_i^{-1},&\quad  \sfS_i(t_j)=t_jt_i^{-\ang{h_i,\al_j}},\\
\sfS_i(f_i) \seteq -e_it_i,&\quad
\sfS_i(f_j) = \sum_{r+s = -\Ang{h_i,\al_j}} (-q_i)^s f_i^{(r)} f_j f_i^{(s)},\\
\sfS_i(e_i) \seteq -t_i^{-1}f_i,&\quad
\sfS_i(e_j) = \sum_{r+s = -\Ang{h_i,\al_j}} (-q_i)^{-r} e_i^{(r)} e_j e_i^{(s)},
\ea\\[1ex]
\ba{ll}
\sfS_i^*(t_i)=t_i^{-1},&\quad  \sfS_i^*(t_j)=t_jt_i^{-\ang{h_i,\al_j}},\\
\sfS_i^*(f_i) \seteq -t_i^{-1}e_i, &\quad \sfS_i^{*}(f_j) = \sum_{r+s = -\Ang{h_i,\al_j}} (-q_i)^r f_i^{(r)} f_j f_i^{(s)},\\
\sfS_i^*(e_i) \seteq -f_it_i,&\quad
\sfS_i^*(e_j) = \sum_{r+s = -\Ang{h_i,\al_j}} (-q_i)^{-s} e_i^{(r)} e_j e_i^{(s)},
\ea
\end{gather}
\end{subequations} 
and 
$\sfS_i^* \circ \sfS_i = \sfS_i \circ \sfS_i^* = {\rm id}$
(see also \cite{Saito94}).  
The automorphisms $\{ \sfS_i\}_{i\in I}$ satisfies the relations of $\ttB_\g$ and hence $\ttB_\g$ acts on $U_q(\g)$ via $\{ \sfS_i\}_{i \in I}$. 

Recall that $\sfS_i$ induces a $\Q(q)$-algebra isomorphism 
$\Uqgm[i] \isoto \Uqgm[i]^*$, 
where
\begin{equation} \label{eq: Uqgmi}
\begin{aligned}
\Uqgm[i] &\seteq \Uqgm \cap\sfS_i^* \Uqgm = \{ x \in \Uqgm \ | \  e_i'(x) = 0 \}, \\
\Uqgm[i]^* &\seteq \Uqgm \cap \sfS_i\Uqgm  = \{ x \in \Uqgm \ | \  e^*_i(x) = 0 \}.
\end{aligned}
\end{equation}
Set $\Uqmz[i]\seteq U_{\bbA}^-(\g)\cap \Uqgm[i]$
and $\Uqmz[i]^*\seteq U_{\bbA}^-(\g)\cap \Uqgm[i]^*$.
They are $\bbA$-subalgebra of
$U^-_q(\g)$.
Then $\sfS_i$ induces an isomorphism
\eq\sfS_i\cl\Uqmz[i]\isoto\Uqmz[i]^*.
\label{eq:Si}
\eneq
Since 
\eq
\Uqmz
=\sum_{k\in\Z_{\ge0}}f_i^{(k)} \Uqmz[i] 
=\sum_{k\in\Z_{\ge0}} \Uqmz[i]^* f_i^{(k)}\label{eq:Wxp}
\eneq 
(\cite[Proposition 3.2.1]{K91}), we have
\begin{equation}
\begin{aligned} \label{eq: Si in A}
\sfS_i(U^-_\bbA(\g)) &=\sum_{k\in\Z_{\ge0}}(e_it_i)^{(k)} \Uqmz[i]^* \qtq \\
\sfS^*_i(U^-_\bbA(\g)) & =\sum_{k\in\Z_{\ge0}}\Uqmz[i](t^{-1}_ie_i)^{(k)}.   
\end{aligned}
\end{equation}

Let us take an element $w$ in $\weyl$.
For a reduced sequence $\uw=(i_1,i_2,\ldots,i_r)$ of $w \in \weyl$ and $1 \le k \le r$,  we set $\be^{\uw}_k \seteq s_{i_1} \ldots s_{i_{k-1}}\al_{i_k}$, 
\begin{equation}
\begin{aligned} \label{eq: PBW}
  \calP_\uw(\be_k) &\seteq \sfS_{i_1} \ldots \sfS_{i_{k-1}} (f_{i_k})\in U^-_\bbA(\g) \qtq \\
  \calP_\uw^\up(\be_k) & \seteq  \zeta_{i_k}\iota\bl\calP_\uw(\be_k)\br.
\end{aligned}
\end{equation}
Note that when $\be_k = \al_i$ for some $i \in I$,
$\calP_\uw(\be_k)=f_i$ and $\calP_\uw^\up(\be_k)$
is equal to $\Ang{i}$ in~\eqref{eq: angi}. 
It is known that   $\calP^\up_\uw(\be_k)$  
belongs to $\Azn$ and  is called the \emph{dual root vector} corresponding to $\be_k$ and $\uw$.
The $\bbA$-subalgebra of $\Azn$ generated by $\{ \calP^\up_\uw(\be_k) \}_{1 \le k \le r}$
does not depend on the choice of a reduced expression $\uw$ of $w$,
which we denote by
$\Aznw$ (see~\cite[Section 4.7.2]{Kimura12}).  We call $A_q(\n(w))
\seteq\Qq\tens_{\Zq}\Aznw$ the \emph{quantum unipotent coordinate ring associated with $w$}. 

\smallskip 

For each $\bsu=(u_1,\ldots,u_r) \in \Z_{\ge0}^{r}$, set 
$$
\bfP^\up_\uw(\bsu) \seteq  \oprod_{k \in [1,r]} q_{i_k}^{u_k(u_k-1)/2} \calP^\up_\uw(\be_k)^{u_k}.  
$$
Then the set $\bfP^\up_\uw \seteq \{ \bfP^\up_\uw(\bsu) \ | \  \bsu \in \Z_{\ge0}^{r} \}$
forms a $\bbA$-basis of  $\Aznw$ and is called the \emph{dual PBW-basis} of $\Aznw$ associated with $\uw$. 

\begin{remark}
Note that there also exists the dual PBW-basis using $\{\sfS_i^{*}\}_{i\in I}$ instead of $\{\sfS_i\}_{i\in I}$, which will be denoted by
$\bfP^{*\up}_{\uw}$. 
Namely, 
\begin{align} \label{eq: PBW *}
\calP^*_\uw(\be_k) \seteq \sfS^*_{i_1} \ldots \sfS^*_{i_{k-1}} (f_{i_k}), \qquad   
  \calP_\uw^{*\up}(\be_k) \seteq 
\zeta_{i_k}\iota\Bigl(\calP^*_\uw(\be_k)\Bigr),
\end{align}
and 
$$
\bfP^{*\up}_\uw(\bsu) \seteq \rprod_{k \in [1,r]} q_{i_k}^{u_k(u_k-1)/2} \calP^{*\up}_\uw(\be_k)^{u_k} \quad\text{for $\bsu \in \Z_{\ge0}^{r}$.}
$$

\end{remark}

\subsection{Crystals and upper global bases} \label{subsec: crystal upper global}
In this subsection, we briefly recall infinite crystals and upper global bases. We refer
\cite{K91, K93, K95, Kashiwarabook} for more details.  

Let $B(\infty)$ be the \emph{infinite crystal} of the negative half $\Uqgm$, and let $\tf_i$ and $\te_i$ be the \emph{crystal operators} on $B(\infty)$.   
For any $b\in B(\infty)$, $\wt(b)$ stands for the weight of $b \in B(\infty) $.  
The $\Q(q)$-algebra anti-involution $*\colon U_q(\g) \to  U_q(\g)$ in~\eqref{eq: star involution} induces the  involution $*$ of $B(\infty)$
and defines another pair of crystal operators $\tf_i^*$ and $\te_i^*$ on $B(\infty)$:
$\tf_i^*b=\bl\tf_i(b^*)\br^*$ and $\te_i^*b=\bl\te_i(b^*)\br^*$.

Let $\overline{\phantom{a}}$ be the $\Q$-algebra anti-automorphism on $\Uqgm$ defined by 
$$
\overline{q}=q^{-1} \qtq \overline{f_i} =f_i. 
$$

Define the map $\sfc \in \End(\Aqn)$ by
\begin{align} \label{eq: sfc}
 \Ang{\sfc(f),x} = \overline{\Ang{f,\overline{x}}} \quad \text{ for any $f\in \Aqn$ and $x \in \Uqgm$.}    
\end{align}
Then it satisfies 
\begin{align} \label{eq: property of sfc}
\sfc(q) =q^{-1}  \qtq \sfc(fg) = q^{(\wt(f),\wt(g))} \sfc(g)\sfc(f) \quad \text{ for } f,g \in \Aqn   
\end{align}
(see~\cite[Proposition 3.6]{Kimura12} for  example). 

Let  
$\bGup \seteq \{ \Gup(b) \ | \ b \in B(\infty) \}$ be the \emph{upper global basis} of $\Azn$ (see \cite{K91, K93, K95} for its definition and properties).
Note that $e'_i\Gup(b)=0$ if $\te_i(b)=0$.

Note that $\Gup(b)$ is $\sfc$-invariant for any $b \in B(\infty)$. 
Set
$$
\LupA \seteq \sum_{b \in B(\infty)} \Z[q]\Gup(b) \subset \Azn.
$$

We regard $B(\infty)$ as a basis of $\LupA/q\LupA$ by 
\begin{align} \label{eq: local basis}
b \equiv \Gup(b) \ {\rm mod} \; q\LupA \quad\text{for $b \in B(\infty)$}.    
\end{align}

We know that $\aform{\Gup(b),\Gup(b')}|_{q=0} =\delta_{b,b'}$ and
hence $B(\infty)$ is an orthonormal basis of
$\LupA/q\LupA$, which implies that
the lattice $\LupA$ is characterized by 
$$
\LupA = \{ x \in \Azn \ | \  \aform{x,x} \in \Z[\hspace{-.4ex}[q]\hspace{-.4ex}] \subset \Q(\hspace{-.4ex}(q)\hspace{-.4ex}) \}.
$$ 

It is proved in~\cite[Theorem 4.29]{Kimura12} that $\bGup$ is \emph{compatible with $\Aznw$}; i.e., 
$$\bGup(w) \seteq \Aznw \cap \bGup \text{ forms a $\bbA$-basis of $\Aznw$.}$$
We set 
$$B(w) \seteq \{  b \in B(\infty) \mid \Gup(b)  \in  \Aznw\}.$$

For $i\in I$, set
 \eqn
&& \Aqn[i]\seteq\iota\bl\Uqgm[i]\br,
 \quad\Aan[i]\seteq\Aqn[i]\cap\Aan\qtq\\
 && \LupA[i]\seteq\Aqn[i]\cap\LupA,\eneqn
 and similarly $\Aqn[i]^*\seteq\iota\bl\Uqgm[i]^*\br$, etc.
 Then we have
 \eqn
 \Aqn[i]=\sum\limits_{b\in \Bi,\;\te_ib=0}\Qq\Gup(b).
\qtq \Aan[i]^*=\sum\limits_{b\in \Bi,\;\te_i^*b=0}\Zq\Gup(b).
 \eneqn
We denote the composition
$\Aqn[i] \To[\iota^{-1}] \Uqgm[i] \To[\sfS_i]  \Uqgm[i]^* \To[\iota] \Aqn[i]^*$ by $\sfS_i$ for simplicity, and define $\sfS_i^*$ on $\Aqn[i]^*$  in a similar manner.

 \Prop\label{prop:iseries}
 Let $i\in I$, $n\in\Z_{\ge0}$ and $x\in\Aqn$ such that $e'_i{}^{n+1}x=0$.
 Then we have
 \bnum
\item there exists a unique sequence $\stt{x_k}_{0\le k\le n}$
  in $\Aqn[i]$ such that
  $$x=\sum_{k=0}^n\ang{i^k}x_k,$$
\item
  if $x$ belongs to $\Aan$ \rmo resp.\ $\LupA$\rmf, then so do
  the $x_k$'s.
 \ee
 \enprop
 \Proof
 It follows from (\cite[Proposition 3.2.1]{K91}) and the fact that
$\Aan$ is stable by $e'_i{}^{(n)}$ for any $n$.
\QED
\Cor\label{cor:iseries}
For any $i\in I$, the multiplication gives  isomorphisms
\eqn
&&\Qq[\ang{i}]\tens_{\Qq}\Aqn[i]\isoto\Aqn,\\
&&\Zq[\ang{i}]\tens_{\Zq}\Aan[i]\isoto\Aan \qtq \\
&&\Bigl(\sum_{k\ge0}\Z[q]\ang{i^k}\Bigr)\tens_{\Z[q]}\LupA[i]\isoto\LupA.
\eneqn
\encor

\section{Bosonic extensions and global bases} \label{sec: bosonic}
In this section, we briefly review the bosonic extensions of 
quantum coordinate rings, which are mainly investigated in~\cite{OP24} for finite types and in~\cite{KKOP24} for arbitrary symmetrizable types (see also~\cite{HL15,FHOO,FHOO2,JLO2}). Let $\sfC=(c_{i,j})_{i,j\in I}$ be a generalized Cartan matrix of symmetrizable Kac-Moody type.
Recall that $\bfk=\Q(q^{1/2})$.

\subsection{Bosonic extensions $\hcalA$} 
The bosonic extension $\hcalA$ of the quantum coordinate ring $\Aqn$   
is the $\bfk$-algebra  generated by the generators
$\{f_{i,p}\mid i\in I,p\in \Z\}$ with the defining relations:
\begin{align}
& \sum_{k=0}^{b_{i,j}} (-1)^k \bnom{ b_{i,j} \\ k}_i
f_{i,p}^kf_{j,p}f_{i,p}^{b_{i,j}-k}=0   \text{  for any $i\ne j\in I$  and $p\in \Z$}, \label{it: def of hA (a)} \allowdisplaybreaks\\
& f_{i,m}f_{j,p}=q^{(-1)^{p-m+1}(\al_i,\al_j)}f_{j,p}f_{i,m}
+\delta_{(j,p),\,(i,m+1)}(1-q_i^2) \ \ \text{ if $m<p$}, \label{it: def of hA (b)}
\end{align}
where $b_{i,j}=1-c_{i,j}$.  We call~\eqref{it: def of hA (a)} the
\emph{$q$-Serre relations}  and~\eqref{it: def of hA (b)} the
\emph{$q$-boson relations}.

We set $\al_{i,m} \seteq (-1)^m\al_i$ for $i \in I$ and $m \in \Z$. Then the relations~\eqref{it: def of hA (a)} and~\eqref{it: def of hA (b)} are homogeneous
by assigning $\wt(f_{i,m})=-\al_{i,m}$. Hence $\hcalA$ admits a weight space decomposition 
$$
\hcalA = \soplus_{\be \in \rl} \hcalA_\be.
$$
We say that an element $x \in \hcalA_\be$ is homogeneous of weight $\beta$,
and   set $\wt(x) \seteq \be$.

There are several (anti-)automorphisms on $\hcalA$ defined as follows (see \cite[\S 3]{OP24} and \cite[\S 4]{KKOP24}):
\begin{eqnarray} &&\left\{
\parbox{75ex}{
\bnum
\item the $\Qh$-algebra anti-automorphism $\star$ 
 defined by $(f_{i,p})^\star = f_{i, -p}$,
\item the $\Q$-algebra anti-automorphism  $\calD$ defined by 
$$\calD (q^{\pm1/2}) = q^{\mp1/2} \qtq \calD(f_{i,p}) = f_{i, p+1},$$ 
\item  the $\Q$-algebra anti-automorphism $\overline{\phantom{a}}$ defined by
$$\overline{q^{\pm1/2}}=q^{\mp1/2} \qtq \overline{f_{i,p}}=f_{i,p},$$
\item the  $\Qh$-algebra  automorphism $\ocalD$  defined by 
$\ocalD(f_{i,p}) = f_{i, p+1}$. 
\ee
}\right. \label{eq: autos}
\end{eqnarray}

We define maps $c,\sigma: \hcalA \to \hcalA$ as follows: For any homogeneous $x \in \hcalA$,
\begin{align} \label{eq: c and sigma}
c(x) \seteq q^{N(\wt(x))} \overline{x} \qtq \sigma(x) \seteq q^{-N(\wt(x))/2}x,
\end{align}
where 
\begin{align}\label{eq: N}
N(\al) \seteq (\al,\al)/2 \quad \text{ for $\al \in \rl$.}
\end{align}
By the definition, we have
\begin{align} \label{eq: c property}
c(xy) =q^{(\wt(x),\wt(y))} c(y)c(x).    
\end{align}

Note that $N(\wt(x))/2 \in \Z/2$. It is easy to see that
(1) $\sigma \circ c = \overline{\phantom{a}} \circ  \sigma$ and
(2) $\sigma$ sends $c$-invariant elements to bar-invariant elements.

\begin{definition}
For $-\infty \le a \le b \le \infty$, let $\hcalA[a,b]$ be the $\bfk$-subalgebra
of $\hcalA$ generated by $\{ f_{i,k} \ | \  i \in I, a\le k \le b\}$.
We simply write 
$$
\hcalA[m] \seteq \hcalA[m,m], \quad 
\hcalA_{\ge m} \seteq \hcalA[m,\infty], \quad
\hcalA_{\le m} \seteq \hcalA[-\infty,m]. 
$$
Similarly, we set $\hcalA_{>m} \seteq \hcalA_{\ge m+1}$
and $\hcalA_{<m} \seteq \hcalA_{\le m-1}$. 
\end{definition}

\begin{theorem}[{\cite[Corollary 4.4]{KKOP24}}] \label{thm: existence}
\hfill 
    \bnum
    \item
    For $m \in \Z$, 
    we have an isomorphism
    $\Uqhm \seteq \Qh \tens_{\Qq} \Uqgm\isoto\hcalA[m]$
    by $f_i\mapsto f_{i,m}$.
    \item
    For any $a,b \in \Z$ with $a \le b$, the $\Qh$-linear map 
$$
\hcalA[b] \tens_\Qh \hcalA[b-1] \tens_\Qh \cdots \tens_\Qh \hcalA[a+1] \tens_\Qh \hcalA[a] \to \hcalA[a,b]
$$
defined by $x_b \tens x_{b-1} \tens \cdots \tens x_{a+1} \tens x_a
\longmapsto x_bx_{b-1}\cdots x_{a+1}x_a$ is an isomorphism.
\ee
\end{theorem}

From Theorem~\ref{thm: existence}, any element $x$ in $ \hcalA[a,b]$ can be expressed as 
$$
x = \sum_{t} x_{b,t} x_{b-1,t} \cdots x_{a,t},
$$  
where $x_{k,t}$ is a homogeneous element in $\hcalA[k]$ and $t$ runs over a finite set. 

\subsection{Bilinear forms on $\hcalA$} 
For homogeneous elements $x,y \in \hcalA$, we set
\begin{align*}
[x,y]_q \seteq xy - q^{-(\wt\, x,\wt\, y)}yx.     
\end{align*}
and extend it to a bilinear homomorphism
$[\ ,\ ]_q\cl \hA\times\hA\to\hA$. 

Then, for any homogeneous elements $x,y,z \in \hcalA$, we have
\begin{equation} 
  \begin{aligned}
{}[x,yz]_q &= [x,y]_qz +q^{-(\wt\, x,\wt\, y)}y[x,z]_q, \\  
{}[xy,z]_q &= x[y,z]_q +q^{-(\wt\,  y,\wt\, z)}[x,z]_qy.    
\end{aligned}\label{eq: []q1}
\end{equation}

\begin{definition}[{\cite[\S 5]{KKOP24}}]
For any $i \in I$ and $m\in \Z$, let $\rmE_{i,m}$ and $\Es_{i,m}$ be the endomorphisms 
of $\hcalA$ defined by 
\begin{align} \label{eq: Ei Esi}
\rmE_{i,m}(x) \seteq [x,f_{i,m+1}]_q \qtq  \Es_{i,m}(x) \seteq [f_{i,m-1},x]_q     
\end{align}
for $x \in \hcalA$.
\end{definition}

Then  $\rmE_{i,m}$ and $\Es_{i,m}$ satisfy the followings: 
\bna
\item $\Es_{i,m} = \star \circ \rmE_{i,-m} \circ \star\;$,
\item $\rmE_{i,m}(f_{j,m}) = \Es_{i,m}(f_{j,m}) =\delta_{i,j}\zeta_i$,
\item $x f_{i,m+1} =  \rmE_{i,m}(x) + q^{-(\al_{i,m},\wt\, x)} f_{i,m+1}x$,
\item $f_{i,m-1}x  =  \Es_{i,m}(x) + q^{-(\al_{i,m},\wt\, x)} x f_{i,m-1}$.
\ee 

Theorem~\ref{thm: existence} says that 
$\hcalA$ admits the decomposition
\begin{align} \label{eq: hcalA decomposition}
\hcalA = \soplus_{(\be_k)_{k \in \Z} \in \rl^{\oplus \Z}} \oprod_{k \in \Z} \hcalA[k]_{\be_k}.     
\end{align}

We define 
$$
\Mn \colon \hcalA\longepito \Qh
$$
to be the natural projection
$ \hcalA\longepito \displaystyle\oprod_{k\in\Z}\hcalA[k]_{0}\simeq\Qh$ deduced from~\eqref{eq: hcalA decomposition}.

\begin{definition}
We define a bilinear form $\hAform{ \ , \ }$ on $\hcalA$ as follows:
\begin{align} \label{eq: hA form}
\hAform{x,y} \seteq \Mn(x \ocalD(y)) \in \Qh \quad \text{ for any } x,y \in \hcalA.   
\end{align}
\end{definition}

\begin{theorem}[{\cite[\S 5]{KKOP24}}] \label{thm: hAform}
The bilinear form $\hAform{ \ , \ }$ is  non-degenerate and symmetric. Furthermore the form $\hAform{ \ , \ }$ satisfies the following properties:
\bna
\item $\hAform{x,y} = \hAform{\ocalD(x), \ocalD(y)} = \hAform{ y^\star,x^\star}$ for any $x,y \in \hA$.
\item $\hAform{f_{i,m}x, y } =  \hAform{x, y f_{i,m+1} }$ and $\hAform{xf_{i,m}, y } =  \hAform{x, f_{i,m-1} y }$ for any $x,y \in \hA$.
\item For any $x,y\in\hcalA_{\le m}$ and $u, v \in\hcalA_{\ge m}$, we have
$$
\hAform{f_{i,m}x,y}=\hAform{ x,\rmE_{i,m}(y)} \qtq \hAform{u , vf_{i,m} }=\hAform{\Es_{i,m}(u) ,v}.
$$
\item $(x,y)=0$ if $x,y$ are homogeneous elements such that
  $\wt(x)\not=\wt(y)$. 
\item For $x = \displaystyle\oprod_{k \in [a,b]} x_k$ and $y = \displaystyle\oprod_{k \in [a,b]} y_k$ with $x_k,y_k \in \hcalA[k]$, we have
$$
\hAform{x,y} = q^{\sum_{s<t} (\wt(x_s),\wt(x_t))} \prod_{k\in [a,b]} \delta(\wt(x_k)=\wt(y_k)) \hAform{x_k,y_k}. 
$$
\ee
\end{theorem}

Using $\sigma$ in~\eqref{eq: c and sigma} and $N$ in~\eqref{eq: N}, we define another bilinear form $\pair{ \ , \ }$ on $\hcalA$ as follows:
\begin{align} \label{eq: pair form}
\pair{x,y} \seteq \hAform{\sigma(x),\sigma(y)} = q^{-N(\wt(x))}\hAform{x,y} \ \  \text{ for any homogeneous } x,y \in \hcalA.  
\end{align}

Note that 
\begin{align*}
\hAform{f_{i,p},f_{i,p}}=1-q_i^2 \qtq \pair{f_{i,p},f_{i,p}} = q_i^{-1}-q_i,    
\end{align*}
and more generally
\begin{align*}
\hAform{f_{i,p}^n,f_{i,p}^n} =\prod_{k=1}^n (1-q_i^{2k}) \qtq \pair{f_{i,p}^n,f_{i,p}^n} = q_i^{-n^2} \prod_{k=1}^n (1-q_i^{2k}).     
\end{align*}

For each $m \in \Z$, we define a $\Q(q)$-algebra homomorphism  
\begin{align} \label{eq: vphi m}
\varphi_{m} \colon \Aqn  \longrightarrow \hcalA[m] \qquad \text{ by } \varphi_m(\Ang{i}) = q_i^{1/2} f_{i,m}, 
\end{align}
and set 
\begin{align} \label{eq: psi}
\psi_m \colon \Uqgm \isoto[\iota] \Aqn \underset{\vph_m}{\longrightarrow} \hcalA.   
\end{align} 
Note that $\varphi_{m}$ induces a $\Qh$-algebra isomorphism between 
$\Aqhn$ and $\hcalA[m]$. 

\begin{proposition}[{\cite[Proposition 5.6]{KKOP24}}] \label{prop: pair}
  The pairing $\pair{\ ,\ }$ have the following properties.
\bnum
\item \label{it: pair (i)}
The bilinear form $\pair{\ , \ }$ is symmetric and non-degenerate.
\item \label{it: pair (ii)}
For 
 $x=\displaystyle\oprod_{k\in[a,b]}x_k$ and
$y=\displaystyle\oprod_{k\in[a,b]}y_k$ with $x_k,y_k\in\hcalA[k]$,
we have
$$
\pair{x,y} = \prod_{k\in[a,b]} \pair{x_k,y_k}.
$$
\item \label{it: pair (iii)} 
For any $x,y\in \Aqn$ and $m\in\Z$, we have 
$$
\aform{x,y} = \pair{\varphi_m(x), \varphi_m(y) }.
$$
\ee
\end{proposition}

\subsection{Extended crystals $\hBi$} Recall the infinite crystal $B(\infty)$ in Section~\ref{subsec: crystal upper global}. The \emph{extended crystal}, introduced in~\cite{KP22,Park23}, is defined
as 
\begin{align} \label{Eq: extended crystal}
	\hB(\infty) \seteq   \Bigl\{  (b_k)_{k\in \Z } \in \prod_{k\in \Z} B(\infty)  \biggm | b_k =\mathsf{1} \text{ for all but finitely many $k$}  \Bigr\},
\end{align}
where $\mathsf{1}$ is the highest weight element of $B(\infty)$. The \emph{extended crystal operators} on $\hB(\infty)$ are defined by the usual crystal operators $\tf_i$, $\te_i$, $\tf_i^*$ and $\te_i^*$, We refer \cite{KP22, Park23}
for details. 
\bna 
\item The involution 
${}^\star \colon \hB(\infty) \rightarrow \hB(\infty)$ is defined as follows:    
\begin{align*} 
\bfb^\star = (b_k')_{k\in \Z} \qquad \text{for any $\bfb = (b_k)_{k\in \Z} \in \hB(\infty)$},
\end{align*}
where $b_k' \seteq (b_{-k})^*$   for any $ k\in \Z$. 
\item 
The bijection
$\ocalD \colon \hB(\infty) \rightarrow \hB(\infty)$ is defined as follows:  
\begin{align*} 
\ocalD(\bfb) = (b_k')_{k\in \Z}  \qquad \text{for any $\bfb = (b_k)_{k\in \Z} \in \hB(\infty)$},
\end{align*}
where $b_k' \seteq b_{k-1}$   for any $ k\in \Z$. 
\ee

\subsection{Global bases of $\hcalA$} \label{subsec: global hatA}
Using the isomorphism $\vph_k$ $(k\in \Z)$, we define
$$
\hcalA_{\bbA}[k] \seteq \vph_{k}(\Azn) \subset \hcalA.
$$
We also define 
$$
\hcalA_\bbA[a,b] \seteq \oprod_{k \in [a,b]} \hcalA_\bbA[k]\subset \hcalA,  \qquad 
\hcalA_\bbA \seteq \sbcup_{a \le b} \hcalA_{\bbA}[a,b]\subset \hcalA. 
$$

\begin{proposition}[{\cite[Proposition 6.2]{KKOP24}}]\label{prop: hcalA integral form}\hfill
\bnum
\item $\hcalA_\bbA$ is a $\bbA$-subalgebra of $\hcalA$, and $\Qh \tens_\bbA \hcalA_\bbA \isoto \hcalA$. 
\item We have
\begin{align} \label{eq: hatA z characterization} 
\hcalA_{\bbA} = \left\{ x \in \hcalA \ \biggm| \ \pair{x,y} \in \bbA \text{ for any } y \in  \displaystyle\oprod_{m \in \Z} \psi_m (U_{\bbA}^-(\g)) \right \}.
\end{align}
\item $\hcalA_\bbA$ is invariant by $c$.
\ee    
\end{proposition}

We define  the $\Z[q]$-lattices
\begin{align*}
\hAzq[k]  
\seteq \vph_k\left(\Lup\big(\Azn\big)\right),  \quad 
\hAzq[a,b]  \seteq  \oprod_{k \in [a,b]} \hAzq[k]
\end{align*}
and  
\begin{align*}
\LuphA & \seteq  \sbcup_{a \le b}  \hAzq[a,b]. 
\end{align*}
Note that they are not closed by multiplication. 

For any $\bfb = (b_k)_{k \in\Z} \in \hB(\infty)$, we set
\begin{align} \label{eq: rmP}
\rmP(\bfb) \seteq \oprod_{k \in \Z} \vph_k(\Gup(b_k)) \in \LuphA\,, 
\end{align}
which forms a $\Z[q]$-basis of $\LuphA$. 

We regard $\hB(\infty)$ as a $\Z$-basis of $\LuphA/q \LuphA$ by 
$$
\bfb \equiv \rmP(\bfb) \qmodLuphA.
$$

For   $\bfb =(b_k)_{k\in \Z} ,\bfb' = (b'_k)_{k\in \Z} \in \hB(\infty)$,   we define  
\begin{align} \label{Eq: order on hB}
  \bfb\preceq\bfb' \quad \text{  if  $\dVert{\wt(b_k)}\le\dVert{\wt(b'_k)}$ for any $k \in \Z$.} 
\end{align}
Note that $\preceq$ in \eqref{Eq: order on hB} is a preorder on $\hB(\infty)$.

Now let us recall one of the main theorems in~\cite{KKOP24} developing the global bases of $\hcalA$. 

\begin{theorem}[{\cite[Theorem 6.6]{KKOP24}}]\label{Thm: global basis} 
\hfill 
\bnum
\item \label{it: global (i)}
For each $\bfb = (b_k)_{k\in\Z}  \in\hB(\infty)$, there exists a unique
$\rmG(\bfb)\in \LuphA$ such that
\begin{align}
\rmG(\bfb)-\rmP(\bfb) &\in \displaystyle\sum_{\bfb'\prec\bfb}q\Z[q]\rmP(\bfb'),  \label{eq: G and P} \\
c(\rmG(\bfb))& =\rmG(\bfb). \label{eq: G and c}
\end{align}
\item \label{it: global (iii)} For each $\bfb = (b_k)_{k\in\Z}  \in\hB(\infty)$,
  $\rmG(\bfb)$
  is a unique element $x \in \LuphA $
  such that
$$c(x)=x\qtq \bfb \equiv x \qmodLuphA.$$ 
\item The set $\{ G(\bfb) \mid\bfb\in \hBi \}$ forms a $\Z[q]$-basis of $\LuphA$, and a $\Z$-basis of $\LuphA \cap c\big(\LuphA\big)$.
 \item  
 For any $\bfb \in\hB(\infty)$, we have 
\begin{align} \label{eq: rmP rmG}
 \rmP(\bfb) = \rmG(\bfb) + \sum_{\bfb'  \prec \bfb}  f_{\bfb,\bfb'}(q) \rmG(\bfb')\qquad \text{for some $f_{\bfb,\bfb'}(q) \in q \Z[q]$.}    
\end{align}
\ee
\end{theorem}

We call $$\bfG  \seteq \{\rmG(\bfb)   \mid   \bfb\in\hBi \} \text{ the \emph{global basis} of $\hcalA$}.$$ 
We also define 
  $$
  \trmG(\bfb) \seteq \sigma( \rmG(\bfb)) \qquad \text{ for any $\bfb \in \hBi$,}
  $$
  and call $  \tbfG \seteq \{\trmG(\bfb)   \mid  \bfb\in\hBi \}$ the \emph{normalized global basis} of $\hcalA$.
Note that $\bfG$ is $c$-invariant while $\tbfG$ is bar-invariant. 

\smallskip

\begin{proposition}[{\cite[Proposition 6.9]{KKOP24}}]\label{Prop: properties of Gb} 
  The $($normalized$)$ global basis has the following properties.
\bnum
\item For $ \bfb, \bfb'\in\hBi$, we have
$$ \pair{\rmG(\bfb),\rmG(\bfb')} = \hAform{ \trmG(\bfb),\trmG(\bfb') } \in\Z[[q]]\cap\Q(q).$$
\item For $\bfb, \bfb'\in\hB(\infty)$, we have 
  $$\pair{ \rmG(\bfb),\rmG(\bfb')} \big\vert_{q=0}=\hAform{ \trmG(\bfb),\trmG(\bfb')} \big\vert_{q=0}=\delta_{\bfb, \bfb'}.$$\label{item:orth}
\item \label{it: pesudo invaraint}  We have
\begin{equation} \label{eq: pesudo invaraint G}
\begin{aligned}
& \{ x \in \hcalA_\bbA \mid c(x) = x, \ \pair{x,x} \in 1 + q\Q[[q]] \}   \allowdisplaybreaks\\
& \hspace{8ex} = \{ \rmG(\bfb) \ | \ \bfb \in \hBi\} \cup \{ -\rmG(\bfb) \ | \ \bfb \in \hBi\}.
\end{aligned}
\end{equation}
  \item 
    We have
    $$\LuphA=\set{x\in\hAz}{\pair{x,x}\in\Q[[q]]}.$$
    \label{item:Lup}
\ee 
\end{proposition}
Note that \eqref{item:Lup}
follows from
the fact that
$\stt{\rmG(\bfb)}_{\bfb\in\hBi}$ is a $\Zq$-basis of $\hAz$,
a $\Z[q]$-basis of $\LuphA$ and \eqref{item:orth}.

\begin{proposition}  [{\cite[Proposition 6.8]{KKOP24}}] \label{prop: star on crystal}
For any $\bfb \in \hB(\infty)$, we have 
$$
\rmG(\bfb)^\star = \rmG(\bfb^\star) \qtq \trmG(\bfb)^\star = \trmG(\bfb^\star).
$$	
\end{proposition}

Note that 
\begin{align} \label{eq: ocalD on crystal}
\ocalD \rmG( \bfb ) = \rmG( \ocalD(\bfb) )    
\end{align}
by the construction of the global basis.

\medskip
\Def\label{def:cble}
Let $\bfb_1,\bfb_2 \in\hBi$.
If there exists $\bfb\in\hBi$ such that
$$\rmG(\bfb_1)\rmG(\bfb_2)\equiv\rmG(\bfb) \qmodLuphA,$$
then we say that $\bfb_1$ and $\bfb_2$ are \emph{\cble} 
or $\rmG(\bfb_1)$ and $\rmG(\bfb_2)$ are  \emph{\cble}, and
we denote $\bfb$ by 
\begin{align} \label{eq: star between crystal}
    \bfb \seteq \bfb_1*\bfb_2.
\end{align}
Note that such an element  $\bfb \in \hBi$ is unique if it exists,
\edf

For $m\in\Z$, set
\eq
&&\ba{l}
\hBi_{\ge m}=\set{\bfb =(b_k)_{k\in\Z}\in\hBi}{\text{$b_k=\one$ for $k<m$}},\\
\hBi_{\le m}=\set{\bfb =(b_k)_{k\in\Z}\in\hBi}{\text{$b_k=\one$ for $k> m$}}.
\ea
\eneq
Then
the subalgebra $(\hcalA_\bbA)_{\ge m} \seteq \hcalA_{\ge m} \cap \hcalA_\bbA$ of
$\hcalA_\bbA$ has a $\bbA$-basis 
$$\bfG_{\ge m} \seteq \{ \rmG(\bfb)  \ | \ \bfb\in\hBi_{\ge m} \}.$$
A similar statement holds for $\hcalA_{\le m}$, etc.

\Prop\label{prop:cble}
Let $m\in\Z$. 
\bnum
\item The multiplication induces an isomorphism
  $(\LuphA)_{\ge m}\tens_{\Z[q]} (\LuphA)_{<m}
  \isoto \LuphA$.
\item $\stt{\rmG(\bfb )}_{\bfb\in\hBi_{\ge m}}$
    is a $\Z[q]$-basis of $(\LuphA)_{\ge m}$,
and $\stt{\rmG(\bfb)}_{\bfb\in\hBi_{\le m}}$
    is a $\Z[q]$-basis of $(\LuphA)_{\le m}$.
\item For any $\bfb\in\hBi$ 
there exists a unique pair of $\bfb'\in\hBi_{\ge m}$ and $\bfb''\in\hBi_{<m}$
such that
$$\rmG(\bfb)\equiv\rmG(\bfb')\rmG(\bfb'') \qmodLuphA. $$
\item Conversely, 
  $\bfb'\in\hBi_{\ge m}$ and $\bfb''\in\hBi_{<m}$
 are \cble
  for any $m\in\Z$. 
\ee
\enprop
\Proof
(i)--(iii) are obvious by the definition.
Let us show (iv).
Set $\bfb'=(b'_k)_{k\in\Z}$
and $\bfb''=(b''_k)_{k\in\Z}$.
Then $\bfb'_k=\one$ for $k<m$ and $b''_k=\one$ for $k\ge m$.
Set $\bfb=(b_k)_{k\in\Z}$ with
$b_k=b'_k$ for $k\ge m$ and
$b_k=b''_k$ for $k< m$.
Then we have
\eqn
\rmG(\bfb)&&\equiv \oprod_{k\in\Z} \vph_k(\Gup(b_k))  \qmodLuphA,\\
\rmG(\bfb')&&\equiv \oprod_{k\ge m} \vph_k(\Gup(b_k))  \qmodLuphA,\\
\rmG(\bfb'')&&\equiv \oprod_{k< m} \vph_k(\Gup(b_k))  \qmodLuphA.
\eneqn
Hence we obtain the desired result.
\QED

\section{Braid symmetries on $\hcalA$} \label{sec: braid}
In this section, we develop the braid symmetries on $\hcalA$, which were treated in~\cite{KKOP21B,JLO2} only for finite types. 
Then, we shall investigate how these symmetries act on the bilinear forms and the global bases of $\hcalA$. 

\subsection{Braid group action on $\hcalA$} 

For each $i \in I$ and $m \in \Z$, we set 
$$
\kappa_i \seteq q_i^{-1/2}(1-q_i^2) \qtq \calF_{i,m} \seteq \kappa_i^{-1} f_{i,m} = q_i^{1/2}(1-q_i^2)^{-1} f_{i,m}. 
$$
Note that $\calF_{i,m} = \psi_m(f_i)$ with $\psi_m$ in~\eqref{eq: psi} and 
$$
\pair{\vph_m(\Ang{i}),\calF_{i,m} } = \pair{\vph_m(\Ang{i}),\psi_{m}(f_i) } =1. 
$$

\emph{The goal of this subsection is to prove the following theorem.}

\begin{theorem} \label{thm: T-braid} For each $i \in I$,
  there exists unique $\bfk$-algebra automorphisms
  $\TT_i$ and $\TT_i^\star$ on $\hcalA$ such that
\begin{subequations} \label{eq: Braid group actions}
\begin{gather}
 \quad \ \ \TT_i(f_{j,m})   \seteq \bc 
\qquad \qquad  f_{i,m+1} & \text{ if } j =i, \\[1ex]
\displaystyle\sum_{r+s = -\Ang{h_i,\al_j}} (-q_i)^s \calF_{i,m}^{(r)} f_{j,m} \calF_{i,m}^{(s)}  & \text{ if } j \ne i, 
\ec \label{eq: T_i} \allowdisplaybreaks\\
\quad  \TTiv{i}(f_{j,m})  \seteq \bc 
\qquad \qquad f_{i,m-1} & \text{ if } j =i, \\[1ex]
\displaystyle\sum_{r+s = -\Ang{h_i,\al_j}} (-q_i)^r \calF_{i,m}^{(r)} f_{j,m} \calF_{i,m}^{(s)}  & \text{ if } j \ne i, 
\ec  \label{eq: T_i inverse}
\end{gather} 
\end{subequations}
where $\calF_{i,m}^{(n)} \seteq \calF_{i,m}^{ \hspace{0.2ex} n}/[n]_i!$ for $n \in \Z_{\ge0}$  and $\calF_{i,m}^{(n)}=0$ for $n<0$. 
Moreover, we have 
\bnum
\item $\TT_i \circ \TTiv{i} =\TTiv{i} \circ \TT_i = \ {\rm id}$,
\item $\{ \TT_i \}_{i \in I}$ and $\{ \TT^\star_i \}_{i \in I}$ satisfy the braid relations.     
  \ee
\end{theorem}

\begin{proof}
  (A)\ Let us first show that $\TT_i$ is a well-defined $\bfk$-algebra endomorphism
  of $\hA$.

  Note that
  $\stt{\calF_{j,m}}_{(j,m)\in I\times\Z}$ satisfies the relations:
  \eq
  &&
  \sum_{r+s=b_{j,k}}(-1)^s\calF_{j,m}^{(r)}\calF_{k,m}\calF_{j,m}^{(s)}=0\qt{for $j\not=k$ where $b_{j,k}=1-\ang{h_j,\al_k}$,}\label{rel:qs}\\
  &&[\calF_{j,m},\calF_{k,p}]_q=\delta(j=k)\delta(p=m+1)
  q_j(1-q_j^2)^{-1}\qt{for any $j,k\in I$ and $m<p$.}\label{rel:mp}
  \eneq
  Set
  \eqn
  Q_{j,m}=\bc \calF_{i,m+1}&\text{if $j=i$,}\\
  \displaystyle\sum_{r+s = -\Ang{h_i,\al_j}} (-q_i)^s \calF_{i,m}^{(r)} \calF_{j,m} \calF_{i,m}^{(s)}  & \text{if $j \ne i$.}
  \ec
  \eneqn
  We have
  $$\wt(Q_{j,m})=s_i\bl\wt(\calF_{j,m})\br.$$
  
  In order to see that $\TT_i$ is well-defined,
  it is enough to show that
  $\stt{Q_{j,m}}_{(j,m)\in I\times\Z}$ satisfies the same relations as
  \eqref{rel:qs}
  and \eqref{rel:mp}.
  
  Set $\calE_j \seteq\sfS_j(f_j)= -e_jt_j\in U_q(\g)$.

The relation ~\eqref{eq: quantum group rel 2}  rewrites as  
\eq\label{eq: kappi boson}
f_j \calE_k = q^{(\al_j,\al_k)}\calE_k f_j +\delta_{j,k} q_j(1-q_j^2)^{-1}(1-t_j^2).
\eneq

Note that $\{ \calE_j \}_{j\in I}$ also satisfies the $q$-Serre relations in~\eqref{eq: quantum group rel 3}. Let us denote by $\tU_\bfk^+(\g)$ the 
$\bfk$-subalgebra of $U_\bfk(\g) \seteq \bfk \tens_{\Qq}U_q(\g)$ generated by $\{ \calE_{j} \}_{j\in I}$. Then we have an isomorphism 
$$U_\bfk^+(\g) \seteq \bfk \tens_{\Qq}U^+_q(\g) \isoto \tU_\bfk^+(\g)
\qt{sending $e_j \mapsto \calE_j$.}$$

\smallskip

Let $A$ be the $\bfk$-subalgebra $U_\bfk(\g)$ generated by
$\set{t_j^2}{j\in I}$, $U_\bfk^-(\g)$ and $\tU_\bfk^+(\g)$. Then, as a $\bfk$-vector space, we have
$$
A \simeq \bfk[t_j^2\mid j \in I] \tens_{\bfk} \tU_\bfk^+(\g) \tens_{\bfk} U_\bfk^-(\g). 
$$
Note that $$\sfS_i\bl U_\bfk^-(\g)\br \subset A.$$

Then $\sum_{j\in I}At_j^2$ is a two-sided ideal of $A$.
By \eqref{eq: kappi boson}, we can define a $\bfk$-algebra homomorphism
$\hA[m,m+1]\to A/\bl\sum_{j\in I}At_j^2\br$ by $\calF_{j,m}\mapsto f_j$
and $\calF_{j,m+1}\mapsto\calE_j$.
Since
$A/\bl\sum_{j\in I}At_j^2\br\simeq \tU_\bfk^+(\g) \tens_{\bfk} U_\bfk^-(\g)$,
the homomorphism above is an isomorphism.
It means that we have a $\bfk$-algebra homomorphism
$$
K_m \colon A \to \hcalA[m,m+1] 
$$
such that
\begin{align*}
K_m(t_j^2)=0, \quad K_m(\calE_j) = \calF_{j,m+1} \qtq     K_m(f_j) = \calF_{j,m}  \quad \text{ for any $j\in I$.}
\end{align*}

Hence $K_m\circ\sfS_i$ gives a $\bfk$-algebra homomorphism
$$K_m\circ\sfS_i\cl U_\bfk^-(\g)\to\hA[m,m+1].$$
We have
$$K_m\bl\sfS_i(f_j)\br=Q_{j,m}\quad \text{for any $j\in I$.}$$
Since $\stt{f_j}_{j\in I}$ satisfies the $q$-Serre relations,
$\stt{Q_{j,m}}_{(j,m)\in I\times\Z}$ satisfies the same relations as
\eqref{rel:qs}.

Hence it remains to prove the $q$-boson relations
\eq
  &&[Q_{j,m},Q_{k,p}]_q=\delta(j=k)\delta(p=m+1)
  q_j(1-q_j^2)^{-1}\qt{for any $j,k\in I$ and $m<p$.}\label{rel:mpQ}
\eneq

Since $[x,y]_q=0$ for any $p\in\Z$, $x\in\hA_{\le p}$ and $y\in\hA_{\ge p+2}$,
it is enough to prove that
\bna
\item \label{it: boson 1} $[Q_{i,m-1},Q_{j,m}]_q=0$ for $j\not=i$,
\item \label{it: boson 2} $[Q_{i,m-2},Q_{j,m}]_q=0$ for $j\not=i$,
\item \label{it: boson 3} $[Q_{j,m},Q_{k,m+1}]_q=\delta(j=k)q_j(1-q_j^2)^{-1}$ for any $j,k\in I\setminus\stt{i}$.
\ee

\mnoi
\eqref{it: boson 1} is equivalent to
$$\bigl[f_{i,m},K\bigr]_q=0,
$$
where $K=\sum_{r+s=c}(-q_i)^sf_{i,m}^{(r)}f_{j,m}f_{i,m}^{(s)}$ and $c=-\ang{h_i,\al_j}$.
Since $-(\al_i,c\al_i+\al_j)=-\sfd_ic$, we have
\eqn
\bigl[f_{i,m},K\bigr]_q
&&=\sum_{r+s=c}(-q_i)^sf_{i,m}f_{i,m}^{(r)}f_{j,m}f_{i,m}^{(s)}
-q_i^{-c}\sum_{r+s=c}(-q_i)^sf_{i,m}^{(r)}f_{j,m}f_{i,m}^{(s)}f_{i,m}\\
&&=\sum_{r+s=c}(-q_i)^s[r+1]_if_{i,m}^{(r+1)}f_{j,m}f_{i,m}^{(s)}
-\sum_{r+s=c}(-q_i)^sq_i^{-c}[s+1]_if_{i,m}^{(r)}f_{j,m}f_{i,m}^{(s+1)}\\
&&=\sum_{r+s=c+1}(-q_i)^s[r]_if_{i,m}^{(r)}f_{j,m}f_{i,m}^{(s)}
-\sum_{r+s=c+1}(-q_i)^{s-1}q_i^{-c}[s]_if_{i,m}^{(r)}f_{j,m}f_{i,m}^{(s)}.
\eneqn
Since
\eqn
(-q_i)^s[r]_i-(-q_i)^{s-1}q_i^{-c}[s]_i
&&=(-q_i)^s\bl[c+1-s]_i+q_i^{-c-1}[s]_i\br\\
&&=(-q_i)^sq_i^{-s}[c+1]_i=(-1)^s[c+1]_i,
\eneqn
we have
$$\bigl[f_{i,m},K\bigr]_q=[c+1]_i\sum_{r+s=c+1}(-1)^s
f_{i,m}^{(r)}f_{j,m}f_{i,m}^{(s)}=0.$$
Thus, we obtain~\eqref{it: boson 1}.

\mnoi
\eqref{it: boson 2} is equivalent to
$$\bigl[f_{i,m-1},K\bigr]_q=0,
$$
where $K=\sum_{r+s=c}(-q_i)^sf_{i,m}^{(r)}f_{j,m}f_{i,m}^{(s)}$ and $c=-\ang{h_i,\al_j}$.
We can easily see
$$\bigl[f_{i,m-1}, f_{i,m}^{(r)}\,\bigr]_q=q_{i}^{r-1}(1-q_i^2) f_{i,m}^{(r-1)}.$$
Here we understand $f_{i,m}^{(-1)}=0$.
Hence we have
\eqn
\bigl[f_{i,m-1},K\bigr]_q
&&=\sum_{r+s=c}(-q_i)^s
\Bigl([f_{i,m-1},f_{i,m}^{(r)}]_q\,f_{j,m}f_{i,m}^{(s)}\\
&&\hs{15ex}+q_i^{2r}f_{i,m}^{(r)}\;[f_{i,m-1},f_{j,m}]_q\;f_{i,m}^{(s)}
+q_i^{2r-c}f_{i,m}^{(r)}f_{j,m}\;[f_{i,m-1},f_{i,m}^{(s)}]_q\Bigr)\\
&&=(1-q_i^2)\sum_{r+s=c}
\Bigl((-q_i)^sq_i^{r-1}f_{i,m}^{(r-1)}f_{j,m}f_{i,m}^{(s)}
+(-q_i)^sq_i^{2r-c}q_i^{s-1}f_{i,m}^{(r)}f_{j,m}f_{i,m}^{(s-1)}\Bigr)\\
&&=(1-q_i^2)\sum_{r+s=c}
\Bigl((-1)^sq_i^{c-1}f_{i,m}^{(r-1)}f_{j,m}f_{i,m}^{(s)}
+(-1)^sq_i^{c-1}f_{i,m}^{(r)}f_{j,m}f_{i,m}^{(s-1)}\Bigr)\\
&&=(1-q_i^2)\sum_{r+s=c-1}
\Bigl((-1)^{s}q_i^{c-1}f_{i,m}^{(r)}f_{j,m}f_{i,m}^{(s)}
+(-1)^{s+1}q_i^{c-1}f_{i,m}^{(r)}f_{j,m}f_{i,m}^{(s)}\Bigr)=0.
\eneqn

\medskip
Finally, let us prove~\eqref{it: boson 3}.
We shall first show
\eq
K_m\bl\sfS_i(\calE_j)\br=Q_{j,m+1}\qt{for $j\not=i$.}
\label{eq:KSQ}
\eneq
We have
\eqn
Q_{j,m+1}
=\sum_{r+s=c}(-q_i)^s\calF_{i,m+1}^{(r)}\calF_{j,m+1}\calF_{i,m+1}^{(s)}
=K_m\Bigl(
\sum_{r+s=c}(-q_i)^s\calE_{i}^{(r)}\calE_{j}\calE_{i}^{(s)}\Bigr).
\eneqn
Since we have
  $$\calE_j^r=(-1)^rq_j^{r(r-1)}e_i^rt_i^r,$$
we obtain
\eq
\sum_{r+s=c}(-q_i)^s\calE_{i}^{(r)}\calE_{j}\calE_{i}^{(s)}
&&=\sum_{r+s=c}(-q_i)^s(-1)^rq_i^{r(r-1)}e_i^{(r)}t_i^r(-e_jt_j)(-1)^s
q_i^{s(s-1)}e_i^{(s)}t_i^s\label{eq:SE}\\
&&=\sum_{r+s=c}(-1)^{c+1}(-q_i)^s
q_i^{r(r-1)+s(s-1)}e_i^{(r)}t_i^re_jt_je_i^{(s)}t_i^s.\nonumber
\eneq
Since
$$e_i^{(r)}t_i^re_jt_je_i^{(s)}t_i^s
=q_i^{-rc+2rs-sc}e_i^{(r)}e_je_i^{(s)}t_i^ct_j
$$
and
$r(r-1)+s(s-1)-rc+2rs-sc=-c$,
we conclude that \eqref{eq:SE}
is equal to
$$(-1)^{c+1}q_i^{-c}\sum_{r+s=c}(-q_i)^se_i^{(r)}e_je_i^{(s)}t_i^ct_j
=-\sum_{r+s=c}(-q_i)^{s-c}e_i^{(r)}e_je_i^{(s)}t_i^ct_j
=\sfS_i(-e_jt_j).$$
Hence we obtain \eqref{eq:KSQ}.

We are now ready to prove~\eqref{it: boson 3}.
Since $[f_j,\calE_k]_q=\delta_{j,k}q_j(1-q_j^2)^{-1}(1-t_j^2)$, we have
\eqn
[Q_{j,m}, Q_{k,m+1}]_q
&&=\bigl[K_m\circ\sfS_i(f_j), K_m\circ\sfS_i(\calE_k)\bigr]_q
=K_m\circ\sfS_i\Bigl([f_j,\calE_k)]_q\Bigr)\\
&&=K_m\Bigl(\delta_{j,k}q_j(1-q_j^2)^{-1}(1-t_j^2t_i^{-2\ang{h_i,\al_j}})\Bigr)
=\delta_{j,k}q_j(1-q_j^2)^{-1}.
\eneqn

\mnoi
(B) Since $\TTiv{i}=\star\circ\TT_i\circ\star$, we conclude that
$\TTiv{i}$ is a well-defined endomorphism.

\snoi
(C) Let us prove that $\TT_i$ and $\TT_i^\star$ are inverse to each other.
We shall  show
$\TT_i\circ\TTiv{i}=\id_{\hA}$.
Thus it is enough to show that
$$
\TT_i \circ \TT_i^\star(\calF_{j,m})  = \calF_{j,m} 
$$
for any $j \in I$ and $m \in \Z$. When $i=j$, it is obvious.

Assume that $j\not=i$.
Then we have
$$K_m(\sfS^*_i(f_j))=\TT_i^\star(\calF_{j,m}) .$$
Hence, we have
$$\TT_i \circ \TT_i^\star(\calF_{j,m})  =
\TT_i\bl K_m(\sfS^*_i(f_j))\br=
K_m(\sfS_i\circ\sfS^*_i(f_j))=K_m(f_j)=\calF_{j,m}.$$
Thus we have obtained $\TT_i\circ\TTiv{i}=\id_{\hA}$.
By applying $\star$, we obtain 
$\TTiv{i}\circ\TT_i=\id_{\hA}$ as well.

\bnoi
(D) Note that
we have
\eq
K_m(\sfS_i(x))=\TT_i \bl K_m(x)\br\qt{for any $x\in\Uqmh$.}
\label{eq:ST}
\eneq
To show the braid relation of  $\{ \TT_i \}_{i \in I}$,
it is enough to show that
\begin{align} \label{eq: braid claim}
\underbrace{\TT_i \circ \TT_j \circ \cdots}_{m_{i,j}\text{-times}}(\calF_{k,m})
= \underbrace{\TT_j \circ \TT_i \circ \cdots}_{m_{i,j}\text{-times}}(\calF_{k,m}) \quad\text{for $i \ne j$ and $k \in I$,}    
\end{align}
where $m_{i,j}$ is the one given in~\eqref{eq: m_ij}. 

If $k\not=i,j$, \eqref{eq: braid claim} is a straight
consequence of the braid relation of
$\stt{\sfS_i}_{i\in I}$ and \eqref{eq:ST}.

Hence we may assume that $k=i$.
Since the proofs are similar, we only consider the case 
$m_{i,j}=4$.

\snoi
By \eqref{eq:ST}, we obtain
\begin{align*}
&\TT_j \circ \TT_i\circ  \TT_j (\calF_{i,m})  =K_m ( \sfS_j \circ \sfS_i\circ  \sfS_j (f_{i}) )  = K_m (f_{i}) =\calF_{i,m}.
\end{align*} 
Hence
\eqn
&&\TT_i\circ\TT_j \circ  \TT_i \circ \TT_j(\calF_{i,m})
=\TT_i(\calF_{i,m}) = \calF_{i,m+1},
\eneqn
and
\eqn
&\TT_j \circ  \TT_i \circ \TT_j\circ  \TT_i (\calF_{i,m})
=\TT_j \circ  \TT_i \circ \TT_j(\calF_{i,m+1}) = \calF_{i,m+1}.
\eneqn
Thus the assertion follows.
\end{proof}

By \eqref{eq:ST}, we have
\begin{equation} \label{eq: T_i S_i}
\begin{aligned} 
 \TT_i(\psi_m(x)) & = \psi_m(\sfS_i x)    \quad \hspace{0.4ex} \text{ for any } x \in \Uqgm \text{ with } e_i'(x)=0, \\
  \TTiv{i}(\psi_m(x)) & = \psi_m(\sfS^*_i x)    \quad \text{ for any } x \in \Uqgm \text{ with } e_i^*(x)=0.
\end{aligned}
\end{equation}

\begin{lemma} \label{lem: T_i and autos}
For any $i\in I$, we have the followings:
\bnum
\item $\ocalD \circ \TT_i=\TT_i \circ \ocalD$,
\item $\TTiv{i} = \star \circ \TT_i \circ \star$,
\item \label{it: T_i and bar}
$\TT_i \circ \overline{\phantom{a}} = \overline{\phantom{a}} \circ \TT_i$,
\item \label{it: T_i and c}
  $\TT_i \circ c = c \circ \TT_i$,
\item $\wt(\TT_ix)=s_i\wt(x)$ for any homogeneous $x\in\hA$,
  \item
  $\TT_i\hA[a,b]\subset\hA[a,b+1]$ and
  $\TTiv{i}\hA[a,b]\subset\hA[a-1,b]$.
\ee
\end{lemma}

\begin{proof}
One can easily check the assertion by their definitions.     
\end{proof}

\subsection{$\TT_i$-invariance of bilinear forms and lattices} 
In this subsection, we will prove the invariance of the bilinear forms and lattices under the automorphisms $\{ \TT_i \}_{i \in I}$.

\begin{proposition}
For $i \in I$ and $x \in \hcalA$, we have 
$$
\Mn(\TT_i(x)) = \Mn( x). 
$$
\end{proposition}
\begin{proof} 
It is enough to show that
\begin{align*}
\Mn(\TT_i(x)) =0 \ \text{ for any } \be \in (-1)^m\nrl \setminus \{ 0 \} \text{ and } x \in \hcalA_{>m} \cdot \hcalA[m]_\be \cdot \hcalA_{<m}.      
\end{align*}
Write $x = yzw$ with $y\in \hcalA_{>m}$, $z\in \hcalA[m]_\be$ and $w\in \hcalA_{<m}$. Then we have $\TT_i(z) \in \bfk[f_{i,m+1}]\cdot \hcalA[m]$.
Write $\TT_i(z) = \sum_{k \in \Z_{\ge0}} f_{i,m+1}^k z_k$ with $z_k \in \hcalA[m]_{s_i\be -k\al_{i,m}}$. Then we have
$$
\TT_i(x) = \sum_{k \in \Z_{\ge0}} (\TT_iy) f_{i,m+1}^k z_k (\TT_i w).
$$
Since $\TT_i(y)\in \hcalA_{>m}$ and $z_k(\TT_i w) \in \hcalA_{\le m}$
we have $\Mn\big( (\TT_iy) f_{i,m+1}^k z_k (\TT_i w) \big)=0$ for $k>0$. 

When $k=0$, we have $\Mn\big( (\TT_iy) z_k (\TT_i w) \big) =0$
since $\TT_i(y) \in \hcalA_{>m}$, $\TT_i(w) \in \hcalA_{\le m}$ and $z_0 \in \hcalA[m]_{s_i\be}$
with $s_i \be \ne 0$. Hence $\Mn(\TT_i x)=0$. 
\end{proof}

\begin{corollary}\label{Cor:invTT}
For any $i \in I$, the pairings $\hAform{ \ , \ }$ and $\pair{ \ , \ }$ are invariant by $\TT_i$. Namely, we have
$$
\hAform{\TT_i x,\TT_i y}= \hAform{x, y} \qtq \pair{\TT_i x,\TT_i y}= \pair{x, y} \quad \text{ for any } x,y \in \hcalA.  
$$
\end{corollary}

\begin{proof}
We have
$$
\hAform{\TT_i x,\TT_i y}= \Mn(  (\TT_i x) \cdot \ocalD(\TT_i y) ) = \Mn(   \TT_i (x \ocalD (y)) ) = \Mn(x \ocalD (y) ) =\hAform{x,y},
$$
and 
\[
\pair{\TT_ix,\TT_iy} = q^{-N(\wt(\TT_ix))} \hAform{\TT_ix,\TT_i y} = q^{-N(\wt(x))} \hAform{x, y} =\pair{x,y}. \qedhere 
\]
\end{proof}

\Prop\label{prop:lattinv}
The  lattices $\hcalA_{\bbA}$ and $\hcalA_{\bfA}\seteq
\bfA\tens_{\Zq}\hA_{\Zq}$ of $\hcalA$ are invariant under $\TT_i^{\;\pm1}$. 
\enprop
\Proof
By \eqref {eq:Wxp} and \eqref{eq: T_i S_i}, we have
$$
\TT_i \left(\psi_m(\Uzgm) \right) \subset
\sum_{k\in\Z_{\ge0}}\calF_{i,m+1}^{(k)}\psi_{m}(\Uzgm).
$$
Thus we obtain 
$$
\TT_i \left(  \oprod_{m \in \Z} \psi_m(\Uzgm) \right) \subset \oprod_{m \in \Z} \psi_m(\Uzgm). 
$$
The similar argument shows
$$
\TTiv{i} \left(  \oprod_{m \in \Z} \psi_m(\Uzgm) \right) \subset \oprod_{m \in \Z} \psi_m(\Uzgm),  
$$
and hence we obtain
$$
\TT_i \left(  \oprod_{m \in \Z} \psi_m(\Uzgm) \right)=\oprod_{m \in \Z} \psi_m(\Uzgm). 
$$
Then \eqref{eq: hatA z characterization} along with Corollary~\ref{Cor:invTT}
implies that
$\TT_i\bl\hAz\br=\hAz$.
\QED

\begin{lemma} \label{lem: Ti LuphA}
The $\Z[q]$-lattice $\LuphA$ of $\hcalA$ is invariant under $\TT_i$. 
\end{lemma}

\begin{proof}
  The assertion follows from
  Proposition~\ref{Prop: properties of Gb}\;\eqref{item:Lup},
  Corollary~\ref{Cor:invTT}
  and Proposition~\ref{prop:lattinv}.
\end{proof}

\subsection{$\TT_i$-invariance of global bases}
In this subsection, we will prove the invariance of global bases of $\hcalA$ under the automorphisms $\{ \TT_i \}_{i \in I}$.

\smallskip

In \cite{Saito94}, Saito proved the invariance of the upper global basis $\bfG^\up$ under the braid group action via $\{ \sfS_i,\sfS^*_i \}_{i \in I}$. 
We briefly review the invariance. Recall the crystal operators $\te_i$, $\tf_i$, $\te_i^*$ and $\tf_i^*$ in Section~\ref{subsec: crystal upper global}.
For $b \in B(\infty)$, we set
$$
\te_i^{\max}(b) \seteq    \te_i^{\upve_i(b)}(b) \qtq  \te_i^{*\max}(b) \seteq    \te_i^{*\upve^*_i(b)}(b),
$$
where $\upve_i(b) \seteq \max\{ k \ge 0 \ | \  \te_i^k(b) \ne 0 \}$
and $\upve^*_i(b) \seteq \max\{ k \ge 0 \ | \  \te_i^{*k}(b) \ne 0 \}$. 

For any $i \in I$, we set
$$
B(\infty)[i] = \{ b \in B(\infty) \ | \ \upve_i(b) = 0 \}
\qtq
B(\infty)[i]^* = \{ b \in B(\infty) \ | \ \upve^*_i(b) = 0 \}. 
$$

The \emph{Saito crystal reflections} on the crystal $B(\infty)$ are defined as follows:
\begin{align} \label{eq: saito reflection}
\scrS_i: B(\infty)[i] \to B(\infty)[i]^* \qtq \scrS^*_i: B(\infty)[i]^* \to B(\infty)[i]     
\end{align}
such that
$$
\scrS_i(b) = \tf_i^{\upvp_i^{*}(b)}\te_i^{*\upve_i^{*}(b)}(b) \qtq \scrS^*_i(b) = \tf_i^{*\upvp_i(b)}\te_i^{\upve_i(b)}(b), \text{ respectively.}
$$
Here $\upvp_i(b) \seteq \upve_i(b)+\Ang{h_i,\wt(b)}$ and
$\upvp^*_i(b) \seteq \upve^*_i(b)+\Ang{h_i,\wt(b)}$. Then it is proved in~\cite{Saito94, Lusztig96, Kimura16} that
$$ 
\sfS_i(\Gup(b))  = \Gup(\scrS_i(b)) \qtq \sfS^*_i(\Gup(b'))  = \Gup(\scrS^*_i(b')) 
$$
for any $b \in B(\infty)[i]$ and $b' \in B(\infty)[i]^*$, respectively. 

It is known that for any $i \in I$ and $b \in B(\infty)$, we have
(\cite{K93})
\begin{align}\label{eq: b equiv} 
\Gup(b) \equiv \Ang{i^{\upve_i(b)}} \Gup(\te_{i}^{\hspace{0.4ex}\max}b) \mod q\LupA.
\end{align} 
Then~\eqref{eq: Si in A} and Lemma~\ref{lem: Ti LuphA} say that
$$
\TT_i \vph_m(\Gup(b)) \equiv \vph_{m+1}\bl \Ang{i^{\upve_i(b)}} \br
\vph_m(\Gup(\tscrS_i(b)) \mod q\hA_{\Z[q]} 
$$
for any $m \in \Z$.

Here we set
\eq
\tscrS_i\seteq \scrS_i\circ \te_i^{\max}\qtq[{and similarly we set}]
\tscrS_i^*\seteq \scrS_i^*\circ \te_i^{*\;\max}.\label{eq:tildS}
\eneq
Hence for $\bfb =(b_k)_{k\in\Z} \in \hB(\infty)$, Theorem~\ref{Thm: global basis}~\eqref{it: global (i)} says that
\begin{align*}
\TT_i(\rmG(\bfb)) & \equiv \oprod_{m\in \Z}      \vph_{m+1}\bl \Ang{i^{\upve_i(b_m)}}\br \vph_m \bl \Gup(\tscrS_i(b_m) \br \\
& \equiv \oprod_{m\in \Z} \vph_m\bl \Gup(\tscrS_i(b_m) \br \vph_{m}\bl \Ang{i^{\upve_i(b_{m-1})}}\br  \qmodLuphA.
\end{align*}

Since $\upve_i^*(\tscrS_i( b_m))=0$ by~\eqref{eq: saito reflection}, we have
$$
\Gup \bl \tscrS_i(b_m) \br   \Ang{i^{\upve_i(b_{m-1})}}  \equiv
\Gup \bl \tf_i^{* \upve_i(b_{m-1})}  \tscrS_i(b_m) \br \mod q\LupA.
$$

Thus we obtain the following theorem.

\begin{theorem} \label{thm: braid crystal}
For $i \in I$, the bases $\bfG$ and $\tbfG$ are invariant under the actions $\TT_i$ and $\TTiv{i}$. More precisely, for  
$\bfb =(b_k)_{k\in\Z} \in \hB(\infty)$, we have
$$
\TT_i \bl \rmG(\bfb) \br = \rmG(\bfb') \qtq \TTiv{i} \bl \rmG(\bfb) \br = \rmG(\bfb''),
$$
where $\bfb' =(b'_k)_{k\in\Z} \in \hB(\infty)$ and  $\bfb'' =(b''_k)_{k\in\Z} \in \hB(\infty)$ are given by
$$b_m' = \tf_i^{* \upve_i(b_{m-1})}  \tscrS_i(b_m) \qtq b_m'' = \tf_i^{\upve_i^*(b_{m+1}) }  \tscrS_i^*(b_m) $$ 
for each $m \in \Z$, respectively \rmo see \eqref{eq:tildS}\rmf.
\end{theorem}
 
\subsection{Braid symmetries on extended crystals}
\label{subsec:braidecr}
Based on Theorem~\ref{thm: braid crystal}, for  $i \in I$,  we shall
define operators $\sfR_i$, $\sfR^\star_i $ on $\hB(\infty)$ 
which are introduced in~\cite{Park23}:
For $\bfb =(b_k)_{k\in\Z} \in \hB(\infty)$, 
\begin{equation}
\begin{aligned}
\sfR_i(\bfb) &\seteq (b'_k)_{k\in\Z}, \quad \ \text{by $b_k' = \tf_i^{* \upve_i(b_{k-1})}  \tscrS_i(b_k)$,} \\    
\sfR^\star_i(\bfb) &\seteq (b''_k)_{k\in\Z}, \quad \text{ by $b_k'' = \tf_i^{\upve^*_i(b_{k+1})}  \tscrS^*_i(b_k)$.}
\end{aligned}
\end{equation}
Hence Theorem~\ref{thm: braid crystal} can be rewritten as
\eq\TT_i\bl\rmG(\bfb)\br=\rmG\bl\sfR_i(\bfb)\br\qtq
\TTiv{i} \bl \rmG(\bfb) \br =\rmG\bl\sfR_i^\star(\bfb)\br.\label{eq:SaitoG}
\eneq
Hence $\sfR^\star_i $ is the inverse of $\sfR_i$, and
$\{ \sfR_i\}_{i\in I}$ and $ \{ \sfR^\star_i \}_{i\in I}$ satisfy the relations of $\ttB_\g$ since so do $\{ \TT_i,\TT_i^\star\}_{i \in I}$. 
Then we have the affirmative answer to the conjecture in~\cite[Introduction]{Park23} which is proved in~\cite{Park23} only for $\g$ of finite type. 

\begin{corollary} \label{cor: braid among R}
The operators $\{ \sfR_i , \sfR^\star_i \}_{i\in I}$ act on $\hB(\infty)$
satisfying the relations of the braid group $\ttB_\g$. Moreover $ \sfR_i$ and $ \sfR^\star_i$ are the inverse of each other.
\end{corollary}

\begin{proof}
The assertion is a direct consequence of Theorem~\ref{thm: braid crystal}. 
\end{proof}

From Corollary~\ref{cor: braid among R},  $\sfR_\ttb$ and $\sfR^\star_\ttb$ are well-defined for any $\ttb \in \ttB$.

\section{PBW-bases theory for $\hcalA$} 
In this section, we introduce the subalgebra $\hcalA(\ttb)$ for $\ttb \in\ttB^+$
and develop its
 PBW-basis theory using the braid symmetries $\{ \TT_i \}_{\in \in I}$. This algebra can be understood as a bosonic-analogue of $A_q(\n(w))$ of $A_q(\n)$ associated with an element
$w$ of the Weyl group $\weyl$. Then we will show that there exist transition maps between PBW bases and global basis of $\hcalA(\ttb)$ satisfying the
unitriangularity.

For the simplicity of notation, we write  
\eqn
\hcalA_{\Qqh}\seteq\hA.
\eneqn

Thus we have defined $\hA_B\subset\hA$ for
$B=\Qqh,\Zq,\Z[q]$.
The subspace $\hA_B$ is a $B$-subalgebra for $B=\Qqh,\Zq$,
but $\hAl$ is only a $\Z[q]$-submodule of $\hA$.
Note that $\hA_B$ is stable by $\TT_i^{\pm1}$.

We set $\hAB{[a,b]}\seteq\hAB\cap\hA[a,b]$ and
similarly for $\hABp$ and $\hABp[\le m]$.
The multiplication gives an isomorphism
$$\hAB[b,c]\tens_B\hAB[a,b-1]\isoto\hAB[a,b]\quad \text{for $a\le b\le c$.}$$
The $B$-module $\hA_B[a,b]$
has $\bfG[a,b]\seteq\bfG\cap\hA[a,b]$ as a basis.

\subsection{Subalgebras $\hcalA(\ttb)$ and their PBW bases} \label{subsec: PBW}
Note that any sequence $\ii=(i_1,\ldots,i_r)$ corresponds to an element
$\ttb\in \ttB^+$ given by $\ttb= r_{i_1} \cdots r_{i_r}$. 
In this case, we call $\ii$ a \emph{sequence} of $\ttb$ and denote by $\Seq(\ttb)$ the set of all sequences of $\ttb$. 
For a sequence $\ii=(i_1,\ldots,i_r) \in \Seq(\ttb)$, we set
\eq \label{eq: cuspidal}
\ttP^\ii_k = \TT_{i_1}\ldots \TT_{i_{k-1}} \vph_0(\Ang{i_k}) \quad 1 \le k\le r,
\eneq
and call it the \emph{cuspidal element} of $\ii$ at $k$.
For $m\in\Z_{\ge0}$, we set
\eq
\diPP{\ii}{m}{k}
\seteq\TT_{i_1}\ldots \TT_{i_{k-1}} \vph_0(\Ang{i_k^m})
=q^{m(m-1)/2}_{i_k}(\ttP^\ii_k )^m.
\label{def:cuspp}
\eneq
When there is no danger of confusion, we drop the superscript
${}^\ii$ for the simplicity of  notation.

\smallskip
Theorem~\ref{thm: T-braid} says that
$\TT_\ttb$ and $\TTiv{\ttb}$ are well-defined
for any element $\ttb \in \ttB^+$. 

\begin{definition}  \label{def: hatA(b)}
For $\ttb\in \ttB^+$, 
we define the $\bfk$-subalgebra of $\hA$ by
$$
\hcalA(\ttb) = (\hcalA)_{\ge 0} \cap \TT_\ttb(\hA_{<0}), 
$$
and a $B$-submodule $\hAB(\ttb)\seteq\hA(\ttb)\cap\hAB$ for $B=\Qqh,\Zq,\Z[q]$.
\end{definition}

We shall prove that the subalgebra $\hcalA(\ttb)$ is generated by the elements $\{ \ttPi{k}\}_{1\le k \le r}$ for \emph{any} $\ii \in \Seq(\ttb)$.  
Note that
\eq
\TT_\ttb(\hAB[a,b])=\TT_\ttb(\hA[a,b])\cap\hAB.
\eneq
Indeed,
we have \eqn
(\TT_\ttb)^{-1}\bl\TT_\ttb(\hA[a,b])\cap\hAB\br
=\hA[a,b]\cap\TT_\ttb^{-1}(\hAB)=
\hAB[a,b].
\eneqn

Let us remark the following elementary lemma, which is used frequently.
\Lemma\label{lem:int}
Let $C$ be a ring,
and let $X'\subset X$ be right $C$-modules, and let $Y'\subset Y$ be left
$C$-modules.
Assume that $X,Y$ are flat and that either
$X/X'$ or $Y/Y'$ is flat.
Then we have
\bnum
\item
$ \xymatrix{
  X'\tens_C Y'\akete[-.7ex]\ar@{^{(}->}[r]\ar@{^{(}->}[d]&X'\tens_CY
  \akete[-.7ex]\ar@{^{(}->}[d]\\
    X\tens_C Y'\ar@{^{(}->}[r]&X\tens_CY,}
  $

  \vs{.5ex}
  \item $(X\tens_C Y')\cap(X'\tens_CY)=X'\tens_C Y'$.
  \ee
  \enlemma
  \Proof
  Assume that $X/X'$ is flat for example.
  Then we have a commutative diagram with exact rows and exact columns: 
  $$\xymatrix@C=3ex@R=3ex{
    &&0\ar[d]&0\ar[d]\\
    0\ar[r]&X'\tens_CY'\ar[r]\ar[d]&X\tens_CY'\ar[r]\ar[d]
    &(X/X')\tens_CY'\ar[r]\ar[d]&0\\
        0\ar[r]&X'\tens_CY\ar[r]\ar[d]&X\tens_CY\ar[r]\ar[d]
        &(X/X')\tens_CY\ar[r]\ar[d]&0\\
               &X'\tens_C(Y/Y')\ar[r]\ar[d]&X\tens_C(Y/Y')\ar[r]\ar[d]
    &(X/X')\tens_C(Y/Y')\ar[r]\ar[d]&0,\\
    &0&0&0&}
  $$\
  which implies the desired result.
  \QED

For $i\in I$, set
\eq
&&
\ba{l}(\hAB[0]){[i]}=\stt{x\in\hAB[0]\bigm|\rmE_{i, 0}(x)=0}\qtq\\
(\hAB[0]){[i]}^{\star}=\stt{x\in\hAB[0]\bigm|\Es_{i,0}(x)=0}.
\ea
  \eneq
  Then by \eqref{eq:Si} and \cite[Lemma 5.1]{KKOP24}, $\TT_i$ induces an isomorphism
  $$\TT_i\cl(\hAB[0]){[i]}\isoto(\hAB[0]){[i]}^{\star}.$$
  
In the rest of this subsection, $B=\Qqh$, $\Zq$, or $\Z[q]$.
\Lemma\label{lem:TTI0}
For any $i\in I$, we have
\eqn
&&\ba{l}
\TT_i(\hAB[0])=\Bigl(
\sum_{0\le k}B\vphi_1(\ang{i^k})\Bigr)\cdot(\hAB[0]){[i]}^{\star},\\
\TT_i^{-1}(\hAB[0])=
(\hAB[0]){[i]}\cdot\Bigl(\sum_{0\le k}B\vphi_{-1}(\ang{i^k})\Bigr),\\
\hA[1]\cap\TT_i(\hAB[0])=\sum_{0\le k}B\vphi_1(\ang{i^k}),\\
\hA[0]\cap\TT_i(\hAB[0])=(\hAB[0]){[i]}^{\star},\\
\hA[-1]\cap\TT_i^{-1}(\hAB[0])=\sum_{0\le k}B\vphi_{-1}(\ang{i^k}),\\
\hA[0]\cap\TT_i^{-1}(\hAB[0])=(\hAB[0]){[i]}.\ea
  \eneqn
  \enlemma
\Proof
 Set $S=\sum_{k\ge0}\Z[q]\ang{i^k}$. Then we have
  $S\cdot\LupA[i]=\LupA$ by Corollary~\ref{cor:iseries}.
  Hence we have
  $$\TT_i(\vphi_0(\LupA))
  =\TT_i\bl\vphi_0(S\cdot\LupA[i])\br=
  \vphi_1(S)\cdot(\hA_{\Z[q]}[0]){[i]}^{\star}.$$
  Hence the first equality is obtained when $B=\Z[q]$.
  We can obtain the first equality in the general case by
  applying $B\tens_{\Z[q]}\scbul$.
  We can obtain the second equality from the first by applying $\star$.
  The other identity easily follows from them.
  \QED
  \Lemma\label{lem:TTIab}
  For any $a, b\in\Z$ with $a\le b$, we have
  \eqn\TT_i(\hAB[a,b])&&=\bl\hAB[b+1]\cap\TT_i(\hAB[b])\br
  \cdot\hAB[a+1,b]\cdot \bl\hAB[a]\cap\TT_i(\hAB[a])\br,\\
 \TT_i^{-1}(\hAB[a,b])&&=\bl\hAB[b]\cap\TT_i^{-1}(\hAB[b])\br
  \cdot\hAB[a,b-1]\cdot \bl\hAB[a-1]\cap\TT_i^{-1}(\hAB[a])\br.\eneqn
  Note that we understand $\hAB[m,p]=B$ if $m>p$.
  \enlemma
  \Proof
We shall prove only the first equality, 
since the second can be obtained from the first by applying $\star$. 
  We argue by induction on $b-a$.
  If $b-a=0$ it follows from Lemma~\ref{lem:TTI0}.
  If $b>a$, then we have
  \eqn
  \TT_i(\hAB[a,b])&&=\TT_i(\hAB[b])\cdot\TT_i(\hAB[a,b-1])\\
  &&=\bl\hAB[b+1]\cap\TT_i(\hAB[b])\br\cdot
  \bl\hAB[b]\cap\TT_i(\hAB[b])\br
  \cdot\bl\hAB[b]\cap\TT_i(\hAB[b-1])\br\\
  &&\hs{10ex}\cdot\hAB[a+1,b-1]\cdot
 \bl\hAB[a]\cap\TT_i(\hAB[a])\br.
 \eneqn
 Since we have
 \eqn
&& \bl\hAB[b]\cap\TT_i(\hAB[b])\br\cdot\bl\hAB[b]\cap\TT_i(\hAB[b-1])\br\\
&&\hs{10ex} =\TT_i\Bigl(\bl\hAB[b]\cap\TT_i^{-1}(\hAB[b])\br
\cdot\bl\hAB[b-1]\cap\TT_i^{-1}(\hAB[b])\br\Bigr)\\
&&\hs{10ex} =\TT_i\TT_i^{-1}(\hAB[b])=\hAB[b],
\eneqn
we obtain
\eqn
\TT_i(\hAB[a,b])
  &&=\bl\hAB[b+1]\cap\TT_i(\hAB[b])\br\cdot
  \hAB[b]\cdot\hAB[a+1,b-1]\cdot
  \bl\hAB[a]\cap\TT_i(\hAB[a])\br\\
 & &=\bl\hAB[b+1]\cap\TT_i(\hAB[b])\br\cdot
  \hAB[a+1,b]\cdot
  \bl\hAB[a]\cap\TT_i(\hAB[a])\br. \qedhere
  \eneqn
  \QED
  \Rem
    In Lemma~\ref{lem:TTIab}, we can replace $\cdot$ with $\tens_B$.
  Namely, the multiplication induces an isomorphism
$$\bl\hAB[b+1]\cap\TT_i(\hAB[b])\br
\tens_B\hAB[a+1,b]\tens_B\bl\hAB[a]\cap\TT_i(\hAB[a])\br
\isoto \TT_i(\hAB[a,b]),$$  
etc. The same remark can be applied to the lemma and proposition below.
\enrem
\Lemma
For any $a, b\in\Z\cup\stt{\infty,-\infty}$ with $a\le b$ and $m\in\Z$, we have
\eqn
\TT_i^{\pm1}(\hAB[a,b])&&=\bl\hA_{\ge m}\cap\TT_i^{\pm1}(\hAB[a,b])\br
\cdot\bl\hA_{<m}\cap\TT_i^{\pm1}(\hAB[a,b])\br\qtq\\
\hAB[a,b]&&=\bl\hAB[a,b]\cap\TT_i^{\pm1}\hA_{\ge m}\br
\cdot\bl\hAB[a,b]\cap\TT_i^{\pm1}(\hA_{<m})\br.
\eneqn
\enlemma
\Proof
The first equality easily follows from the preceding lemma,
since for any sequence $\stt{K_k}_{k\in[a,b]}$ of $B$-submodules of $\hAB[k]$
such that $\hAB[k]/K_k$ is flat,
setting $K=\oprod_{k\in[a,b]}K_k$, we have
$$K=(\hA_{\ge m}\cap K)\cdot(\hA_{<m}\cap K).$$
The second is obtained from the first by the
application of $\TT_i^{\mp1}$.
\QED

\Prop\label{prop: hA(b) span}
  Let $\ttb\in\ttB^+$, $\ii =(i_1,\ldots, i_r) \in \Seq(\ttb)$
  and $B=\Qqh,\Zq,\Z[q]$.
\bnum
\item \label{it: step 1 hA(b)}
  For any $i \in I$, we have
  $\TT_i((\hcalA_B)_{<0}) = \bl\sum_{k\ge0}B\vphi_0(\ang{i^k})\br\cdot (\hcalA_B)_{<0}$.
\item \label{it: step 2 hA(b)}
  $\TT_\ttb\bl\hABp[{<0}]\br =
  \Bigl(\TT_\ttb\bl\hABp[{<0}]\br\cap \hABp[\ge0]\Bigr)\cdot\hABp[{<0}]$.
\item \label{it: step 3 hA(b)}
  $\TT_\ttb(\hABp[{<0}]) \cap \hABp[\ge0]= \oprod_{1\le k\le r}\bl\sum_{m\ge0}B\ttP^{\ii,(m)}_k\br$.
\ee
\enprop

\begin{proof}

\snoi
\eqref{it: step 1 hA(b)} follows from Lemma~\ref{lem:TTI0}
and Lemma~\ref{lem:TTIab}.

\snoi
\eqref{it: step 2 hA(b)}--\eqref{it: step 3 hA(b)}  
We shall first show
\eq
\TT_{\ttb }(\hABp[<0]) =
\Bigl(\oprod_{1\le k\le r}\bl\sum_{m\ge0}B  \diPP{\ii}{m}{k}  \br\Bigr)\cdot\hABp[<0]
\eneq
by induction on $r$. Set $\ii'=(i_2,\ldots,i_r)$ and $\ttb'=r_{i_2}\cdots r_{i_r} \in \ttB$.
Set
$$
T=\oprod_{1\le k\le r}\bl\sum_{m\ge0}B\diPP{\ii}{m}{k}\br\qtq
T'=\oprod_{1\le k\le r-1}\bl\sum_{m\ge0}B\diPP{\ii'}{m}{k}\br.$$
Then we have
$$
\TT_{\ttb'}(\hABp[<0]) = T'\cdot \hABp[{<0}]
$$
by the induction hypothesis.
Since $T=\bl\TTi{i_1}T'\br\cdot \bl\sum_{k\ge0}B\,\vphi_0(\ang{i_1^k})\br$, we have

\eqn
\TT_{\ttb}(\hABp[{<0}]) &&  = \TTi{i_1}\TTi{\ttb'}(\hABp[{<0}])  
 = \TTi{i_1}(T'\cdot\hABp[{<0}])\\
&&= \TTi{i_1}(T')\cdot\bl\sum_{k\ge0}B\,\vphi_0(\ang{i_1^k})\br\cdot\hABp[{<0}]
  =  T\cdot\hABp[{<0}].
  \eneqn

\noindent
The equality $\TT_{\ttb}\hABp[{<0}])\cap\hABp[{ \ge 0}]=T$
 follows from
$\hAB\simeq \hABp[{\ge0}] \tens \hABp[{<0}]$, Lemma~\ref{lem:int} and the fact that
$\hABp[<0]/B$ is a flat $B$-module.
\QED

Note that we have
\eq
E_{i,m}(f_{i,m}^k)=\Es_{i,m}(f_{i,m}^k) = (1-q_i^{2k})f_{i,m}^{k-1}. 
\label{eq:Ef}
\eneq

Recall $\Ang{i^n}$ in~\eqref{eq: Ang i n}. Then we have 
$\sigma(\vph_m(\Ang{i^n})) = f_{i,m}^n$ and \eqref{eq:Ef} implies that
$$
\hAform{f_{i,m}^n,f_{i,m}^n} = \hAform{ \sigma(\vph_m(\Ang{i^n})) ,\sigma(\vph_m(\Ang{i^n})) } = \pair{\vph_m(\Ang{i^n}),\vph_m(\Ang{i^n})} =  \prod_{k=1}^n (1-q_i^{2k}).
$$

For a sequence $\ii =(i_1,\ldots,i_r) \in \Seq(\ttb)$ and $\bsu=(u_1,\ldots,u_r) \in \Z_{\ge0}^r$, we define the \emph{PBW-element}  
\begin{align} \label{eq: PBW element}
\ttP^\ii(\bsu)  \seteq \oprod_{k \in [1,r]} \diPP{\ii}{u_k}{k}. 
\end{align}
Here $ \diPP{\ii}{n}{k}$ is defined in \eqref{def:cuspp}.
Note that $\diPP{\ii}{n}{k}$ is $c$-invariant, since $\vph_0(\Ang{i_k^{n}})$ is $c$-invariant and $\TT_i$ preserves 
$c$-invariant elements by Lemma~\ref{lem: T_i and autos}~\eqref{it: T_i and c}. 

\begin{proposition} \label{prop: orthogonality}
For a sequence $\ii =(i_1,\ldots,i_r) \in \Seq(\ttb)$ and $\bsu,\bsv \in \Z_{\ge0}^r$, we have
\begin{align}\label{eq: pairing PBW}
\pair{\ttP^\ii(\bsu),\ttP^\ii(\bsv)}  = \delta_{\bsu,\bsv} \prod_{k=1}^r \pair{ \vph_0(\Ang{i_k^{u_k}}),\vph_0(\Ang{i_k^{v_k}}) } 
= \delta_{\bsu,\bsv} \prod_{k=1}^r \prod_{s=1}^{u_k} (1-q_{i_k}^{2s}).  
\end{align}
\end{proposition}

\begin{proof}
We argue by induction on $r$. Set $\ii'=(i_2,\ldots,i_r)$, $\bsu' = (u_2,\ldots,u_r)$ and $\bsv' = (v_2,\ldots,v_r)$. Then we have
\begin{align*}
\ttP^\ii(\bsu) = \left(   \TTi{i_1}(\ttP^{\ii'}(\bsu')) \right) \vph_0(\Ang{i_1^{u_1}}) \qtq
\ttP^\ii(\bsv) = \left(   \TTi{i_1}(\ttP^{\ii'}(\bsv')) \right) \vph_0(\Ang{i_1^{v_1}}).
\end{align*}
Hence we have
\begin{align*}
 \pair{\ttP^\ii(\bsu),\ttP^\ii(\bsv)}   & = \Bpair{\left(   \TTi{i_1}(\ttP^{\ii'}(\bsu')) \right) \vph_0(\Ang{i_1^{u_1}}),\left(   \TTi{i_1}(\ttP^{\ii'}(\bsv')) \right) \vph_0(\Ang{i_1^{v_1}}) } \allowdisplaybreaks\\
  & = \pair{ \ttP^{\ii'}(\bsu')   \TTi{i_1}^{-1} \big( \vph_0(\Ang{i_1^{u_1}}) \big) ,   \ttP^{\ii'}(\bsv')  \TTi{i_1}^{-1} \big( \vph_0(\Ang{i_1^{v_1}}) \big) } 
  \allowdisplaybreaks\\
  & = \pair{ \ttP^{\ii'}(\bsu')     \vph_{-1}(\Ang{i_1^{u_1}})   ,   \ttP^{\ii'}(\bsv')   \vph_{-1}(\Ang{i_1^{v_1}})  } 
  \allowdisplaybreaks\\
  & \underset{*}{=} \pair{ \ttP^{\ii'}(\bsu') , \ttP^{\ii'}(\bsv') } \pair{    \vph_{-1}(\Ang{i_1^{u_1}})   ,     \vph_{-1}(\Ang{i_1^{v_1}})  },
\end{align*}
which implies the desired result. Here $\underset{*}{=}$ follows from the fact that $\ttP^{\ii'}(\bsu'),\ttP^{\ii'}(\bsv') \in \hcalA_{\ge0}$. 
\end{proof}

\begin{corollary} \label{cor: P_i k-basis}
Let $\ttb$ be an element in the braid monoid $\ttB^+$ with $\ell(\ttb)=r$. Then, for any $\ii \in \Seq(\ttb)$, and $B=\Qqh,\Zq,\Z[q]$, the set 
$$
\bfP_\ii \seteq \{ \ttP^\ii(\bsu) \ |  \ \bsu \in \Z_{\ge0}^r \} \ \text{ forms a $B$-basis of $\hAB(\ttb)$.}
$$
\end{corollary}
\begin{proof}
  By Proposition~\ref{prop: hA(b) span}~\eqref{it: step 3 hA(b)}, $\bfP_\ii$ spans
  $\hAB(\ttb)$. Then Proposition~\ref{prop: orthogonality} implies the assertion.      
\end{proof}

We call $\bfP_\ii$ the \emph{PBW-basis of $\hcalA(\ttb)$ associated with $\ii \in \Seq(\ttb)$}.

\subsection{PBW-bases and global bases} \label{subsec: PBW and global}
In this subsection, we investigate the relationship between the global bases and PBW-bases of subalgebras of $\hcalA$.
Let $B=\Qqh,\Zq,\Z[q]$.
Since $\hAB$ and $\TT_{\ttb}(\hAB)$ have a $B$-basis
$\bfG\cap \hAB$ and $\bfG\cap \TT_{\ttb}\bl\hAB\br$,
respectively, $\hAB(\ttb)$ has also a $B$-basis
$$\bfG(\ttb)\seteq\bfG\cap \hcalA(\ttb).$$

\Prop \label{prop:PequivG}
For any sequence $\ii=(i_1,\ldots, i_r)\in I^r$ and $\bsu=(u_1,\ldots,u_r)\in\Z_{\ge0}^r$,
there exists a unique
$\cb(\ii,\bsu)\in\hBi$ such that
$$\ttP^\ii(\bsu)\equiv \rmG\bl\cb(\ii,\bsu)\br\qmodLuphA.$$
Moreover, $\sfR_{i_1}\bl\cb(\ii',\bsu')\br$ and $\vphi_{0}( \ang{i_1^{u_1}} )$
are \cble and
$$\cb(\ii,\bsu)=\sfR_{i_1}\bl\cb(\ii',\bsu')\br*  \vphi_{0}( \ang{i_1^{u_1}} ),$$
where $\ii'=(i_2,\ldots,i_r)$ and $\bsu'=(u_2,\ldots,u_r)$
\rmo see {\rm \S\;\ref{subsec:braidecr}, Definition~\ref{def:cble}} and {\rm Proposition~\ref{prop:cble}}\rmf.
\enprop
\Proof
We shall argue by induction on $r$.
If $r=1$, it is obvious.
Assume that $r>1$.
Then by the induction hypothesis,
we have
$$\ttP^{\ii'}(\bsu')\equiv \rmG\bl\cb(\ii',\bsu')\br\qmodLuphA.$$
On the other hand, we have
$$\ttP^\ii(\bsu)=\TT_{i_1}\bl\ttP^{\ii'}(\bsu')\br\vphi_0\bl\ang{i_1^{u_1}}\br.$$
Hence
we have
\eqn
\TT_{i_1}^{-1}\ttP^\ii(\bsu)&&=
\ttP^{\ii'}(\bsu') \cdot  \vphi_{-1}  \bl \ang{i_1^{u_1}}\br\\
&&\equiv \rmG\bl\cb(\ii',\bsu')\br\cdot\vphi_{-1}  \bl  \ang{i_1^{u_1}}\br\qmodLuphA\\
&&\equiv \rmG\bl\cb(\ii', \bsu')*\vphi_{-1}(\ang{i_1^{u_1}})\br\qmodLuphA
\eneqn
by Proposition~\ref{prop:cble} since
$\cb(\ii',\bsu')\in\hBi_{\ge0}$.
Thus we obtain
\[ \ttP^\ii(\bsu)\equiv \TT_{i_1}\Bigl(
\rmG\bl\cb( \ii', \bsu')*\vphi_{-1}(\ang{i_1^{u_1}})\br\Bigr)
\equiv \rmG\bl\sfR_{i_1}\cb(\ii',\bsu')*\vphi_{0}(\ang{i_1^{u_1}})\br. \qedhere
\]
\QED

\Cor 
Let $\ttb\in\ttB^+$ and $\ii\in\Seq(\ttb)$.
Then we have
\bnum
\item
  $\bfG\cap \hcalA(\ttb)=\set{\rmG(\cb(\ii,\bsu) )}{\bsu\in\Z_{\ge0}^r}$,
\item  
  $\stt{\rmG\bl\cb(\ii,\bsu)\br}_{\bsu\in\Z_{\ge0}^r}$ is a $B$-basis of $\hAB(\ttb)$.
  \ee
  \encor

  \Proof  
  Let $Y$ be the image of $\bfG\cap\hA_{\Z[q]}$ in $\hA_{\Z[q]}/q\hA_{\Z[q]}$.
  Then $Y$ is a $\Z$-basis of $\hA_{\Z[q]}/q\hA_{\Z[q]}$.
 Note that
  both $\bfG\cap \hcalA(\ttb)$ and $\stt{\ttP^\ii(\bsu)\mid\bsu\in\Z_{\ge0}^r}$
  are $\Z[q]$-bases of $\hA_{\Z[q]}(\ttb)$.
  Hence they give $\Z$-bases of
  $\hA_{\Z[q]}(\ttb)/q\hA_{\Z[q]}(\ttb)$.
  Proposition~\ref{prop:PequivG} implies that
the image of $\stt{\ttP^\ii(\bsu)\mid\bsu\in\Z_{\ge0}^r}$
  in $\hA_{\Z[q]}/q\hA_{\Z[q]}$ is equal to 
  $X\seteq\stt{\rmG\bl\cb(\ii,\bsu)\br\mod q\hA_{\Z[q]}\mid \bsu\in\Z_{\ge0}^r}$.
Hence  $X\subset Y$ is a $\Z$-basis of
  $\hA_{\Z[q]}(\ttb)/q\hA_{\Z[q]}(\ttb)\subset \hA_{\Z[q]}/q\hA_{\Z[q]}$.
It means that
$X=Y\cap \bl\hA_{\Z[q]}(\ttb)/q\hA_{\Z[q]}(\ttb)\br$, which implies that
  $\stt{\rmG\bl\cb(\ii,\bsu)\mid\bsu\in\Z_{\ge0}^r}=\bfG\cap
  \hA(\ttb)$.  
\QED

\begin{definition} \label{def:rev}
  Let $\ttb\mapsto \ttb^\rev$ be the anti-automorphism of the group $\ttB$ given by
  $r_i\mapsto r_i$.

  Let $\ttb\in\ttB^+$ and $\ii=(i_1,i_2,\ldots,i_r) \in \Seq(\ttb)$.
\bna
\item We set $\ii^\rev \seteq (i_r,\ldots,i_2,i_1)\in\Seq(\ttb^\rev)$. 
\item For $\bsu=(u_1,\ldots,u_r) \in \Z_{\ge0}^r$, we set $\bsu^\rev \seteq (u_r,\ldots,u_1) \in \Z_{\ge0}^r$. 
\ee
\end{definition}
Then we have 
$$\TT_{\ttb^\rev}^{\star} = (\TT_{\ttb})^{-1}.$$

The following proposition can be understood as a braid-analogue of
\emph{Levendorski\u{\i}-Soibelman $(LS)$ formula} (see \cite[Proposition 5.2.2]{LS91}).

\begin{proposition} \label{prop: LS}
Let $\ii =(i_1,\ldots,i_r) \in \Seq(\ttb)$ for $\ttb \in \ttB^+$.
For  $1 \le k < t \le r$, we have
\begin{align} \label{eq: LS formula}
[\ttP^{\ii}_k,\ttP^{\ii}_t]_q = \sum_{\bsu =(u_{k+1},\ldots,u_{t-1}) \in \Z_{\ge0}^{(k,t)}} Q_{\bsu} \ttP^{u_{t-1}}_{t-1} \cdots \ttP^{u_{k+2}}_{k+2} \ttP^{u_{k+1}}_{k+1}
\quad\text{for some $Q_{\bsu} \in \bbA$.}
\end{align}
\end{proposition}

\begin{proof}
  Since $\{ \TT_i \}_{i\in I}$ are automorphisms, we can assume $k=1$
and $t=r$ without loss of generality. 
  We can write
  \eq
[\ttP^{\ii}_1,\ttP^{\ii}_r]_q = \sum_{\bsu =(u_{1},\ldots,u_{r}) \in\Z_{\ge0}^{[1,r]}} 
Q_{\bsu} \ttP^\ii(\bsu)
\quad\text{for some $Q_{\bsu} \in \bbA$.}\label{eq:PBWcom}
\eneq
Hence it is enough to show that
$Q_{\bsu}=0$ if $u_1 \neq 0$ or $u_r\neq0$.

\snoi
(a)\ We shall first show that $Q_{ \bsu}=0$ if $u_1 \neq 0$.
Note that we have 
\begin{align*}
[\ttP^{\ii}_1,\ttP^{\ii}_t]_q &=  [ \vph_{0}(\Ang{i_1}), \TT_{i_1}\TT_{i_{2}}\ldots \TT_{i_{t-1}} \vph_{0}(\Ang{i_{t}}) ]_q \allowdisplaybreaks    \\
& = \TT_{i_1}  [ \vph_{-1}(\Ang{i_1}),  \TT_{i_{2}}\ldots \TT_{i_{t-1}} \vph_{0}(\Ang{i_{t}}) ]_q.  
\end{align*}
Since $\vph_{-1}(\Ang{i_1}) =q_{i_1}^{1/2}f_{i_1,-1} \in \hcalA_{<0}$ and $\TT_{i_{2}}\ldots \TT_{i_{t-1}} \vph_{0}(\Ang{i_{t}}) \in \hcalA_{\ge0}$, 
we have
$$[ \vph_{-1}(\Ang{i_1}), \TT_{i_{2}}\ldots \TT_{i_{t-1}} \vph_{0}(\Ang{i_{t}}) ]_q \in \hcalA_{\ge0}$$
by the defining relation in~\eqref{it: def of hA (b)}.

We have
$$\TT_{i_1}^{-1}\ttP^\ii(\bsu)=\ttP^{\ii'}(u_2,\ldots,u_r)\;
\vphi_{-1}\bl\ang{i_{1}^{u_1}}\br,$$
and hence
$$\sum_{\bsu =(u_{1},\ldots,u_{r}) \in\Z_{\ge0}^{ [1,r]  } } 
Q_{\bsu}\ttP^{\ii'}(u_2,\ldots,u_r) \vphi_{-1}\bl\ang{i_{1}^{u_1}}\br$$
belongs to $\hA_{\ge0}$.
Therefore, $Q_{\bsu}=0$ if $u_1\neq0$.

\mnoi
(b)\ 
We shall show that $Q_{\bsu}=0$ if $u_r\not=0$.
We can easily check (see\ Corollary~\ref{cor:Theta} below) 
$$\ttP^{\ii^\rev}_{r+1-k}=\TT_{\ii^\rev}\circ\star\circ \db (\ttP^{\ii}_k),$$
where $\ii^\rev=(i_r,\ldots,i_1)$.
Hence we obtain
$$\ttP^{\ii^\rev}(\bsu^\rev)=\TT_{\ii^\rev}\circ\star\circ \db \bl\ttP^{\ii}(\bsu)\br,$$
where $\bsu^\rev=(u_r,\ldots,u_1)$.
Thus, we conclude that $Q_{\bsu}=0$ for $u_r\not=0$
by using (a) and applying the anti-automorphism $\TT_{\ii^\rev}\circ\star\circ\db$
to \eqref{eq:PBWcom}.
\end{proof}

\begin{definition} \label{def: bi-lexico}
We denote by $<_{{\rm bi}}$ the bi-lexicographic  partial order on $\Z_{\ge0}^r$; i.e., for $\bsa=(a_1,a_{2},\ldots,a_r),\bsa'=(a'_1,a'_{2},\ldots,a'_r) \in \Z_{\ge0}^r$,
$\bsa <_{{\rm bi}} \bsa'$ if and only if the following conditions hold:
\bna
\item \label{it: lex 1}  there exists $s \in [1,r]$ such that $a_t = a_t'$ for any $t <  s$ and
  $ a_s < a_s'$,
\item \label{it: lex 2} there exists $u \in [1,r]$ such that $ a_t = a_t' $ for any $t > u$ and
  $a_u < a_u'$.
\ee
\end{definition} 

The following corollary is a direct consequence of LS-formula. 
\begin{corollary}
We have
$$
c(\ttP^\ii(\bsu)) = \ttP^\ii(\bsu) +\sum_{\bsu'   <_{{\rm bi}}  \bsu} f_{\bsu,\bsu'}  \ttP^\ii(\bsu') \quad\text{for some $f_{\bsu,\bsu'} \in \bbA$.}
$$
\end{corollary}

\begin{lemma} \label{lem: Gb}
  There exists a family   $\stt{\rmG_\ttb(\bsu,\ii)}_{\bsu \in \Z_{\ge0}^r }$
  of elements in $\hcalA_\bbA(\ttb)$
such that 
\begin{align}
  c(\rmG_\ttb(\bsu,\ii))   & = \rmG_\ttb(\bsu,\ii), \label{eq: Gb-invaraint} \\
  \ttP^\ii(\bsu) & = \rmG_\ttb(\bsu,\ii) + \sum_{\bsu'  <_{{\rm bi}}  \bsu} g_{\bsu,\bsu'} \rmG_\ttb(\bsu',\ii) \quad\text{for some $g_{\bsu,\bsu'} \in q \Z[q ]$}.
  \label{eq: ttPi b triangular}
\end{align}
\end{lemma}

\begin{proof}
The proof is the same as the one for the existence of Kazhdan-Lusztig polynomials.    
\end{proof}

From the unitriangularity in~\eqref{eq: ttPi b triangular}, we also have
$$
  \rmG_\ttb(\bsu,\ii)    = \ttP^\ii(\bsu)  + \sum_{\bsu'  <_{{\rm bi}}  \bsu} t_{\bsu,\bsu'} \ttP^\ii(\bsu')  \quad\text{for some $t_{\bsu,\bsu'} \in q \Z[q ]$}.
$$

\Lemma\label{lem:ggb}
For any $\bsu\in\Z_{\ge0}^r$, we have
$$\rmG_\ttb(\bsu,\ii) =\rmG(\cb(\ii,\bsu)).$$
\enlemma
\Proof
Indeed, $\rmG_\ttb(\bsu,\ii)\in\hL$,
$\rmG_\ttb(\bsu,\ii)\equiv\ttP^\ii(\bsu)\equiv
\rmG(\cb(\ii,\bsu))\qmodLuphA$,
and $\rmG_\ttb(\bsu,\ii)$ is $c$-invariant.
\QED

\begin{conjecture} \label{conj: quantum positive} 
If $\g$ is of symmetric type, the coefficient $g_{\bsu,\bsu'}$
in~\eqref{eq: ttPi b triangular} is contained in $q\Z_{\ge 0}[q]$.
\end{conjecture}

\Prop
{\rm Conjecture~\ref{conj: quantum positive}} holds for any simply-laced finite type $\g$.    
\enprop

\begin{proof}
Note that (a) $\hcalA_\bfA$ is isomorphic to \emph{the quantum Grothendieck ring $\calK_t(\Cg^0)$} of a \emph{skeleton category} over quantum affine algebra $U_q'(\hg^{(1)})$, introduced in~\cite{Nak04,VV02,Her04}, 
(b) the normalized global basis $\tbfG$ corresponds to the basis $\bfL_t$ of $\calK_t(\Cg^0)$  
consisting of $(q,t)$-character of simple modules by \cite[Theorem 7.4]{KKOP24}, 
whose structure coefficients are contained in $q\Z_{\ge 0}[q]$ proved in  \cite{VV03}  (see also \cite{FHOO,FHOO2}).
Since each $\ttP^\ii_k$ is an element in $\bfG$ by~\eqref{eq: ttPi b triangular} and Lemma~\ref{lem:ggb}, the assertion follows. 
\end{proof}

\section{Applications} \label{sec: Application}
In this section, we investigate subalgebras of $\hcalA$ one step further,
which are also constructed by the braid symmetries $\{ \TT_i \}_{i \in I}$.

\smallskip
Let us first introduce subalgebras $\hcalA$ associated with
a pair of elements in $\ttB^+$:

\begin{definition}
For elements $\ttb,\ttb' \in \ttB^+$ and $B=\Qqh, \Zq,\Z[q]$, we define the subalgebras of $\hcalA_{\ge0}$:
$$ 
\hAB( \star,\ttb) \seteq \TT_\ttb \bl \hABp[{\ge 0}] \br  \  \text{ and } \ 
\hAB(\ttb,\ttb') \seteq  \hAB(\ttb) \cap \hAB(\star,\ttb')
=\TT_\ttb\bl\hABp[{<0}]\br\cap\TT_{\ttb'}\bl\hABp[{\ge0}]\br.
$$
\end{definition} 
 
Similarly to $\hAB(\ttb)$, we have
\bnum
\item $\hAB(\star,\ttb)$ has a $B$-basis
  $\bfG(\star,\ttb) \seteq \bfG\cap \hAB(\star,\ttb)$,
\item $\hAB(\ttb,\ttb')$ has a $B$-basis $\bfG(\ttb,\ttb') \seteq
  \bfG\cap\hA(\ttb,\ttb')$. 
\ee

\subsection{Tensor product decompositions of $\hcalA_{\ge0}$ corresponding to $\ttb$} In this subsection, we show that the multiplications between 
$\hcalA(\star,\ttb)$ and $\hcalA(\ttb)$ induce vector space isomorphisms from $\hcalA(\star,\ttb) \tens \hcalA(\ttb)$ and 
$ \hcalA(\ttb) \tens \hcalA(\star,\ttb)$ to $\hcalA_{\ge0}$, which can be understood as bosonic-analogues of \cite{Kimura16,Tani17} about $A_q(\n(w))$
 for a Weyl group element $w$.

\Prop \label{prop: iso via multiplication}
For any $B=\Qqh,\Zq, \Z[q]$  and $\ttb\in\ttB^+$, we have the following:
\bnum
\item \label{it: Tb}
For any $m\le0$, we have
$$\TT_\ttb\bl\hABp[{<0}]\br=\Bigl(\TT_\ttb(\hABp[{<0}])\cap\hABp[\ge m]\Bigr)\cdot
\hABp[<m].$$
\item \label{it: m le 0}
For any $m\le0$, we have
$$\hABp[{\ge m}]=\TT_\ttb\bl\hABp[{\ge0}]\br\cdot
\Bigl(\TT_\ttb\bl\hABp[{<0}]\br\cap\hABp[\ge m]\Bigr).$$
\ee
In particular, the multiplication induces an isomorphism
\eq\hAB(\star,\ttb)\tens\hAB(\ttb)\isoto\hABp[\ge0].
\label{eq:bbb}
\eneq
\enprop
\Proof
Let us first prove~\eqref{it: Tb}.
By Proposition~\ref{prop: hA(b) span},
we have
$$\TT_\ttb\bl\hABp[{<0}]\br
=\Bigl(\TT_\ttb\bl\hABp[{<0}]\br\cap\hABp[\ge 0]\Bigr)\cdot
\hABp[{<0}].$$
Hence, we have
\eqn
\TT_\ttb\bl\hABp[{<0}]\br
&&=\Bigl(\TT_\ttb(\hABp[{<0}])\cap\hABp[\ge 0]\Bigr)\cdot
\hAB[m,-1]\cdot\hABp[<m].
\eneqn
Set $K=\Bigl(\TT_\ttb(\hABp[{<0}])\cap\hABp[\ge 0]\Bigr)\cdot
\hAB[m,-1]$.
Then Lemma~\ref{lem:int}
implies
$$\TT_\ttb\bl\hABp[{<0}]\br\cap \hABp[\ge m]
=\Bigl(K\tens_B\hABp[<m]\Bigr)\cap\Bigl(\hABp[\ge m]\tens_BB\Bigr)
=K\tens_BB.$$
Hence we obtain~\eqref{it: Tb}.

\medskip
Let us prove~\eqref{it: m le 0}.
By~\eqref{it: Tb}, we have
\eqn
&&\TT_\ttb\bl\hABp[\ge0]\br\cdot\Bigl(\TT_\ttb(\hABp[{<0}])\cap\hABp[\ge m]\Bigr)\cdot
\hABp[<m]\\
&&\hs{10ex}=\TT_\ttb\bl\hABp[\ge0]\br\cdot\TT_\ttb\bl\hABp[<0]\br
=\hAB=\hABp[\ge m]\cdot\hABp[<m].
\eneqn
Set
$K=\TT_\ttb\bl\hABp[\ge0]\br\cdot \left( \TT_\ttb\bl \hABp[{<0}]\br \cap\hABp[\ge m] \right)$.
Then we have
$$K\tens_B\hABp[<m]\isoto\hABp[\ge m]\tens_B\hABp[<m],$$
which implies that $K\isoto\hABp[\ge m]$.
\QED

Applying $c$ to \eqref{eq:bbb}.
we obtain the following corollary.

\begin{corollary}
For any element $\ttb \in \ttB^+$ and $B =\Q(q^{1/2}),\Zq$, the multiplication in $\hcalA$ also gives another isomorphism of $B$-modules
\begin{align} \label{eq: mult b inverse}
\hcalA_{ B}(\ttb) \tens_{B} \hcalA_{B}(\star,\ttb) \longrightarrow (\hcalA_B)_{\ge 0}.
\end{align}
\end{corollary}

\subsection{Bosonic analogue of quantum twist maps} In this subsection, we establish the anti-isomorphism between $\hcalA(\ttb)$
and $\hcalA(\ttb^\rev)$ preserving their PBW-bases and global bases. It can be understood as a bosonic analogue of quantum twist map in~\cite{LY19,KimuOya18}. 

\smallskip 

Let us recall (anti)-automorphisms $\star$ and $\ocalD$ on $\hcalA$ given in~\eqref{eq: autos}. For $\ttb \in \ttB^+$, let us define 
\begin{align}
\Theta_\ttb \seteq \TT_\ttb \circ \star \circ \ocalD
\end{align}
which is a $\bfk$-algebra anti-automorphism on $\hcalA$.
Moreover, $\Theta_\ttb$ induces an automorphism of the global basis $\bfG$.
We can check easily
\eq
\TT_\ttb \circ \star \circ \ocalD = \star \circ \ocalD \circ
  \TT_\ttb^{\star}. \label{eq:theta}
\eneq
Since $\Theta_\ttb(\hA_{<0})=\TT_\ttb(\hA_{\ge0})$ and
  $\Theta_\ttb(\hA_{\ge0})=\TT_\ttb(\hA_{<0})$,
we have
\begin{align*} 
\hcalA(\ttb) = \hcalA_{\ge 0} \cap \Theta_\ttb \bl \hcalA_{\ge 0} \br \qtq 
\hcalA(\star,\ttb) = \Theta_\ttb \bl\hcalA_{< 0}\br. 
\end{align*}

From $(\star \circ\ocalD)^2={\rm id}$,
\eqref{eq:theta} and $\TT_\ttb^\star=(\TT_{\ttb^\rev})^{-1}$
(see Definition~\ref{def:rev}),
we have
$$\Theta_{\ttb^\rev} \circ \Theta_{\ttb}={\rm id}.
$$
Thus $\Theta_{\ttb^\rev}$  induces an isomorphism 
$$
\Theta_{\ttb^\rev}\cl \hAB(\ttb) \isoto \hAB(\ttb^\rev). 
$$

\begin{proposition} \label{prop: cuspidal to cuspidal}
For $\ttb \in \ttB^+$ and $\ii=(i_1,\ldots,i_r) \in \Seq(\ttb)$, we have
$$
\Theta_{\ttb^\rev}(\TT_{i_1} \ldots \TT_{i_{k-1}} \vph_0(\Ang{i_k}) )= \TT_{i_r} \ldots \TT_{i_{k+1}}  \vph_0(\Ang{i_k}) 
\quad\text{for $1 \le k\le r$.}
$$
\end{proposition}

\begin{proof}
It can be easily checked that
$$
\TT_i \circ \star \circ \ocalD (f_{i,m})= f_{i,-m}.
$$
Hence we have 
\begin{align*}
\Theta_{\ttb^\rev}(\TT_{i_1} \ldots \TT_{i_{k-1}} \vph_0(\Ang{i_k}) ) & = 
(\TT_{i_r} \ldots \TT_{i_1} \circ \star \circ \ocalD) (\TT_{i_1} \ldots \TT_{i_{k-1}} \vph_0(\Ang{i_k}) )     \\
& = (\TT_{i_r} \ldots \TT_{i_k} \circ \star \circ \ocalD)(\vph_0(\Ang{i_k})) \\
& = (\TT_{i_r} \ldots \TT_{i_{k+1}})(\vph_0(\Ang{i_k})). \qedhere
\end{align*}
\end{proof}

From Proposition~\ref{prop: cuspidal to cuspidal}, we have the following corollary.

\begin{corollary}\label{cor:Theta}
For $\ttb \in \ttB$, $\ii \in \Seq(\ttb)$ and $\bsu \in \Z_{\ge 0}^r$, we have
$$
\Theta_{\ttb^\rev}(\ttP^\ii(\bsu)) = \ttP^{\ii^\rev}(\bsu^\rev). 
$$
\end{corollary}

Hence the anti-isomorphism $\Theta_{\ttb^\rev}$ sends $\bfG(\ttb)$ to $\bfG(\ttb^\rev)$ bijectively, which can be understood as a bosonic analogue of
\cite[Theorem 3.8]{KimuOya18}:

\Prop \label{prop: twist map}
For $\ttb  \in \ttB^+$, $\ii \in \Seq(\ttb)$  and $\rmG_\ttb(\bsu,\ii) \in \bfG(\ttb)$, we have
$$
\Theta_{\ttb^\rev}\bl\rmG(\cb(\ii,\bsu)\br= \rmG\bl\cb( \ii^\rev,\bsu^\rev )\br. 
$$
\enprop

\begin{proof}
  It follows from
  $$\Theta_{\ttb^\rev}\bl\rmG(\cb(\ii,\bsu)\br
  \equiv\Theta_{\ttb^\rev}(\ttP^\ii(\bsu))= \ttP^{\ii^\rev}(\bsu^\rev)
  \equiv
  \rmG\bl\cb(\ii^\rev,\bsu^\rev)\br
\qmodLuphA.\qedhere$$
\QED

\providecommand{\bysame}{\leavevmode\hbox to3em{\hrulefill}\thinspace}
\providecommand{\MR}{\relax\ifhmode\unskip\space\fi MR }
\providecommand{\MRhref}[2]{%
  \href{http://www.ams.org/mathscinet-getitem?mr=#1}{#2}
}
\providecommand{\href}[2]{#2}

\end{document}